\documentclass{siamart1116}

\usepackage{amsfonts,multirow}
\usepackage{amssymb,latexsym,amsmath,enumerate,verbatim,amsfonts,algorithmic,algorithm,cite}
\usepackage{bm}
\usepackage{color} 
\usepackage[active]{srcltx}

\ifpdf
  \DeclareGraphicsExtensions{.eps,.pdf,.png,.jpg}
\else
  \DeclareGraphicsExtensions{.eps}
\fi

\numberwithin{theorem}{section}

\newcommand{\TheTitle}{Subgradient-based approach for MFS$_C$ problem}
\newcommand{\TheAuthors}{Minglu Ye and Ting Kei Pong}

\headers{\TheTitle}{\TheAuthors}

\title{A subgradient-based approach for finding the maximum feasible subsystem with respect to a set}

\author{
Minglu Ye \thanks{Department of Mathematics and Information, China West Normal University, Nanchong, Sichuan, China.
\email{yml2002cn@aliyun.com}.}
\and Ting Kei Pong \thanks{Department of Applied Mathematics, the Hong Kong Polytechnic University, Hong Kong.
This author was supported partly by Hong Kong Research Grants Council PolyU153005/17p. \email{tk.pong@polyu.edu.hk}.}
}

\usepackage{amsopn}

\ifpdf
\hypersetup{
  pdftitle={\TheTitle},
  pdfauthor={\TheAuthors}
}
\fi

\newtheorem{example}{Example}

\newcommand{\inner}[2]{\langle#1,#2\rangle}
\newcommand{\norm}[1]{\|#1\|}

\def\R{{\rm I\!R}}
\def\F{{\frak M}}

\def\R{{\mathbb R}}

\def\dist{{\rm dist}}

\def\argmin{\mathop{\rm arg\,min}}

\def\EAS{{EAS$_{\mbox{\tiny{\rm MFS}$_C$}}$}}

\def\GP{{sGP$_{ls}$}}

\begin{document}

\maketitle

\begin{abstract}
  We propose a subgradient-based method for finding the maximum feasible subsystem in a collection of closed sets with respect to a given closed set $C$ (MFS$_C$). In this method, we reformulate the MFS$_C$ problem as an $\ell_0$ optimization problem and construct a sequence of continuous optimization problems to approximate it. The objective of each approximation problem is the sum of the composition of a nonnegative nondecreasing continuously differentiable concave function with the squared distance function to a closed set. Although this objective function is nonsmooth in general, a subgradient can be obtained in terms of the projections onto the closed sets. Based on this observation, we adapt a subgradient projection method to solve these approximation problems. Unlike classical subgradient methods, the convergence (clustering to stationary points) of our subgradient method is guaranteed with a {\em nondiminishing stepsize} under mild assumptions. This allows us to further study the sequential convergence of the subgradient method under suitable Kurdyka-{\L}ojasiewicz assumptions. Finally, we illustrate our algorithm numerically for solving the MFS$_C$ problems on a collection of halfspaces and a collection of unions of halfspaces, respectively, with respect to the set of $s$-sparse vectors.
\end{abstract}



\section{Introduction}

Let $\{C,D_1,\ldots,D_m\}$ be a collection of finitely many nonempty (possibly nonconvex) closed sets in $\R^n$. We consider the problem of finding the {\em maximum feasible subsystem with respect to $C$} (MFS$_C$):
\begin{equation}\label{MFSC}
\begin{array}{rl}
  \max & \#(I)\\
  {\rm s.t.}& C\cap \bigcap_{i\in I}D_i\neq \emptyset,\\
  & I \subseteq \{1,\ldots,m\},
\end{array}
\end{equation}
where $\#(I)$ stands for the cardinality of the index set $I$. The above problem is a natural generalization of the widely studied problem of finding the {\em maximum feasible (linear) subsystem} (MF$\ell$S), which is just \eqref{MFSC} with $C= \R^n$ and $D_i$ being halfspaces and is known to be NP hard; see \cite{chakravarti94}. The MF$\ell$S problem finds applications in many different areas such as image and signal processing \cite{Amaldi1999}, operations research \cite{Amaldi1995,Amaldi1998}, machine learning \cite{Amaldi03} and linear programming \cite{Chinneck96,Chinneck01,Greenberg91}, and various solution methods have been proposed. Many of these methods are based on integer programming techniques and exploit explicitly the fact that each $D_i$ is a halfspace and $C= \R^n$; see, for example, \cite{Parker95,PRyan96,Pfetsch02,Pfetsch08} and references therein. For instance, the recently proposed branch-and-cut method in \cite{Pfetsch08} builds on the classical branch-and-cut approach for integer programming: it exploits the duality between MF$\ell$S and the problem of finding the minimum {\em irreducible infeasible subsystem} (IIS) cover, and makes use of the structure of a special kind of polytope to identify IIS covers; see Sections~3.1 and 3.2 of \cite{Pfetsch08}. Thus, when it comes to the MFS$_C$ problem \eqref{MFSC}, it is not clear whether the method in \cite{Pfetsch08} and other existing methods for the MF$\ell$S problem can be readily generalized to solve \eqref{MFSC} for general sets $C$ and $D_i$, which are possibly nonconvex.

In this paper, we develop a new approach for approximately solving the MFS$_C$ problem \eqref{MFSC}. Our method takes advantage of the recent advancement in $\ell_0$ minimization such as \cite{Zhao12}, and is based on the observation that the MFS$_C$ problem \eqref{MFSC} is equivalent to the following nonlinear programming problem with an $\ell_0$ objective:
\begin{align}\label{sum_zero_norm}
\min_{x\in C}\ \Phi_0(x) := \sum_{i=1}^{m}|d_{D_i}^2(x)|_0
\end{align}
where $|\cdot|_0$ is the $\ell_0$ norm.\footnote{This equivalence can be easily deduced by noting that if $I_*$ solves \eqref{MFSC}, then the solution set of \eqref{sum_zero_norm} is $C\cap \bigcap_{i\in I_*}D_i$, and that if $x^*$ solves \eqref{sum_zero_norm}, then a solution $I_*$ of \eqref{MFSC} is given by $I_* = \{i:\; \dist (x^*,D_i) = 0\}$.} In our approach, as in \cite{Zhao12}, we approximate the $\ell_0$ norm in $\Phi_0$ by a sequence of continuous functions. We show that if the sequence of continuous functions $\{\varphi_{\epsilon_k}\}$ is chosen in such a way that it is {\em both} epi-convergent and pointwise convergent to the $\ell_0$ norm, then the sequence of functions
\begin{equation}\label{Phi_epsilon}
\Phi_{\epsilon_k}(x) := \sum_{i=1}^{m}\varphi_{\epsilon_k}(d_{D_i}^2(x))
\end{equation}
epi-converges and pointwise converges to $\Phi_0$. We then explore how to minimize $\Phi_{\epsilon_k}$ over $C$, under further differentiability and concavity assumptions on $\varphi_{\epsilon_k}$ (see Section~\ref{sec4} for the assumptions and concrete examples of $\{\varphi_{\epsilon_k}\}$). Notice that the function $\Phi_{\epsilon_k}$, though continuous, is still possibly nonsmooth in general: this is because the function $x\mapsto d_{D_i}^2(x)$ is differentiable if and only if $D_i$ is convex. Fortunately, a subgradient of the squared distance function to $D_i$ can be obtained in terms of the projections onto $D_i$.
We thus propose a subgradient projection method for minimizing the $\Phi_{\epsilon_k}$ in \eqref{Phi_epsilon} over $C$. Surprisingly, we are able to show that the projected subgradient direction is indeed a {\em descent direction}. This enables us to incorporate the standard nonmonotone line-search scheme to empirically accelerate the algorithm. Moreover, under mild assumptions on the collection of closed sets, we show that the stepsizes used are uniformly bounded away from zero and that any accumulation point of the sequence generated by our subgradient projection method is a stationary point of $\Phi_{\epsilon_k} + \delta_C$.\footnote{As a corollary, under mild assumptions on the collection of closed sets, we establish that the sequence generated by the averaged projection method clusters at stationary points of a suitable potential function; see Corollary \ref{coro:aver}.} Based on these and some suitable Kurdyka-{\L}ojasiewicz (KL) assumptions, we further show that the whole sequence generated by our method (with monotone linesearch) for minimizing the $\Phi_{\epsilon_k}$ in \eqref{Phi_epsilon} over $C$ is convergent to a stationary point of $\Phi_{\epsilon_k}+\delta_C$ when each $D_i$ is convex or $C = \R^n$. We also establish a relationship between the different KL assumptions considered. Finally, we perform numerical experiments on \eqref{MFSC} under two different scenarios: we consider $C$ being the set of $s$-sparse vectors, and $\{D_1,\ldots,D_m\}$ being either a collection of halfspaces or a collection of unions of halfspaces. Our experiments on random instances show that our approach is able to identify a reasonably large feasible subsystem with respect to $C$ in a reasonable period of time, even for large-scale problems.

The rest of this paper is organized as follows. We introduce notation and preliminary materials in Section~\ref{sec2}. An approximation scheme for solving \eqref{MFSC} based on approximately minimizing a bunch of $\Phi_{\epsilon_k}$ in \eqref{Phi_epsilon} over $C$ is introduced in Section~\ref{sec3}. We then propose and study in Section~\ref{sec4} a subgradient method for minimizing $\Phi_{\epsilon_k}$ in \eqref{Phi_epsilon} over $C$ and show that the sequence generated clusters at a stationary point of $\Phi_{\epsilon_k}+\delta_C$ under mild assumptions on the collection of closed sets and some further differentiability and concavity assumptions on $\varphi_{\epsilon_k}$. Sequential convergence is studied in Section~\ref{sec5} under additional KL assumptions. Finally, numerical experiments are presented in Section~\ref{sec6}.
%
%
%
%
%

\section{Notation and preliminaries}\label{sec2}

We let $\R^n$ denote the $n$-dimensional Euclidean space and let $\inner{\cdot}{\cdot}$ and $\norm{\cdot}$ denote the standard inner product and the induced norm, respectively. The nonnegative orthant and positive orthant are denoted by $\mathbb{R}^n_{+}$ and $\mathbb{R}^n_{++}$, respectively. For an $x\in \R^n$, we let $\|x\|_0$ denote the $\ell_0$ norm of $x$, and let $\mathbf{B}(x,r)$ denote the closed ball centered at $x$ with radius $r$. Moreover, for an $s\in \R$, we use $|s|_0$ to denote its $\ell_0$ norm and $[s]_+ := \max\{s,0\}$ to denote its positive part.

Let $C\subseteq \R^n$ be a nonempty closed set. We let $P_C(x)$ denote the set of projections of a vector $x\in \R^n$ onto $C$: this set is always nonempty, and is a singleton when $C$ is in addition convex. The distance to $C$ from $x$ is denoted by $\dist (x,C)$ or $d_C(x)$, and we use $C^\infty$ to denote the horizon cone of $C$, which is defined as
$C^{\infty}:=\{x: \exists~ x^t\in C,\ \lambda_t\downarrow 0 \text{~with~} \lambda_t x^t\rightarrow x\}$.\footnote{We recall from \cite[Theorem~3.5]{Rock_wets98} that $C^\infty = \{0\}$ if and only if $C$ is bounded.}
Finally, we let $\delta_C$ denote the indicator function of $C$, which is defined as
\begin{equation*}
\delta_C(x)=
\begin{cases}
0       & \text{~if~}x\in C,\\
+\infty & \text{otherwise}.
\end{cases}
\end{equation*}

For an extended-real-valued function $f:\R^n\to [-\infty,\infty]$, we let ${\rm dom}\,f:=\{x\in \R^n:\; f(x) < \infty\}$. Such a function is said to be proper if ${\rm dom}\,f\neq \emptyset$ and $f$ is never $-\infty$, and is said to be closed if $f$ is lower semicontinuous. For a proper closed function $f$, the regular subgradient and (limiting) subgradient \cite[Definition 8.3]{Rock_wets98} at a point $\bar{x}\in {\rm dom}\,f$ are defined respectively as
\begin{align*}
&\widehat{\partial} f(\bar{x}): = \left\{v :\; \liminf_{x\rightarrow \bar{x},\,x\neq \bar{x}}\frac{f(x)-f(\bar{x})-\inner{v}{x-\bar{x}}}{\norm{x-\bar{x}}}  \geq 0 \right\},\\
&\partial f(\bar{x}): = \left\{v :\; \exists x^t\xrightarrow{f}\bar x \text{~and~} v^t\in \widehat{\partial} f(x^t)\text{~with~}v^t\rightarrow v\right\},
\end{align*}
where $x^t\xrightarrow {f} \bar x $ means both $x^t\rightarrow  \bar x $ and $f(x^t)\rightarrow f(\bar x)$. We define $\partial f(x) = \widehat\partial f(x) = \emptyset$ whenever $x\notin {\rm dom}\,f$ by convention, and write ${\rm dom}\partial f:= \{x\in \R^n:\; \partial f(x)\neq \emptyset\}$. Clearly, we have $\widehat{\partial} f(\bar{x}) \subseteq \partial f(\bar{x})$. It is known that $\partial f$ reduces to the classical subdifferential in convex analysis if $f$ is in addition convex \cite[Proposition~8.12]{Rock_wets98}, and we have $\partial f(x) = \{\nabla f(x)\}$ if $f$ is continuously differentiable at $x$. We also define
the normal cone of a nonempty closed set $C$ at $x \in C$ as $N_C(x) = \partial \delta_C (x)$.
Finally, for a positive number $\nu$, we denote the set of $\nu$-minimizers of $f$ by $\nu$-$\argmin f$, i.e., $\nu$-$\argmin f:= \{x:f(x)\leq \inf f + \nu\}$.
The set of minimizers of $f$ is denoted by $\argmin f$.

We next recall the Kurdyka-{\L}ojasiewicz (KL) property, which is an important property for analyzing convergence of first-order methods; see, for example, \cite{AttBolt09,AtBoReSo10,Attouch-Bolte13}. For notational simplicity, for any $\nu\in (0,\infty]$, we let $\Xi_\nu$ denote the set of continuous concave functions $\phi:[0,\nu)\to \R_+$ that are continuously differentiable on $(0,\nu)$ with positive derivatives and satisfy $\phi(0) = 0$.

\begin{definition}[{{\bf KL property \& KL function}}]\label{KL-exponenct}
Let $f$ be a proper closed function. We say that $f$ satisfies the Kurdyka-{\L}ojasiewicz (KL) property at $\hat x\in {\rm dom}\,\partial f$ if there exist a neighborhood $\cal N$ of $\hat x$, $\nu \in (0,\infty]$ and a $\phi\in \Xi_\nu$ such that
\begin{equation}\label{haha}
      \phi'(f(x) - f(\hat x))\,\dist (0,\partial f(x))\ge 1
\end{equation}
whenever $x\in {\cal N}$ and $f(\hat x)< f(x) < f(\hat x) + \nu$. If $f$ satisfies the KL property at $\hat x\in {\rm dom}\,\partial f$ and the $\phi$ in \eqref{haha} can be chosen as $\phi(s) = cs^{1-\theta}$ for some $\theta\in [0,1)$ and $c > 0$, then $f$ is said to satisfy the KL property at $\hat x$ with exponent $\theta$.

A proper closed function $f$ is called a KL function if it satisfies the KL property at every point in ${\rm dom}\,\partial f$, and is called a KL function with exponent $\theta\in[0,1)$ if it satisfies the KL property with exponent $\theta$ at every point in ${\rm dom}\,\partial f$.
\end{definition}

It can be shown that the KL property is satisfied by a large class of functions, including all proper closed semialgebraic functions. Indeed, it is known that proper closed semialgebraic functions are KL functions with exponent $\theta$ for some $\theta\in [0,1)$; see, for example, \cite[Section~4]{AtBoReSo10} and references therein. We next recall the following lemma concerning a {\em uniformized} KL property, first proved in \cite[Lemma~6]{BolteSabach14}. It was used there for establishing convergence of first-order methods for level-bounded objective functions.

\begin{lemma}[{{\bf Uniformized KL property}}]\label{uni-kl}
Let $\Omega$ be a compact set and $f$ be a proper closed function that satisfies the KL property at every point in $\Omega$ and is constant on $\Omega$, say, equals $l_*$. Then there exist $\epsilon>0$, $\nu \in (0,\infty]$ and a $\phi\in \Xi_\nu$ such that
\begin{equation*}
      \phi'(f(x) - l_*)\,\dist (0,\partial f(x))\ge 1
\end{equation*}
whenever $d_\Omega(x) < \epsilon$ and $l_* < f(x) < l_* + \nu$.
\end{lemma}

%
%

Finally, we prove that a certain sequence of function is equi-lsc. This will be used in Section~\ref{sec3} to construct an explicit example of sequence $\{\Phi_{\epsilon_t}\}$ (as in \eqref{Phi_epsilon}) that epi-converges and pointwise converges to $\Phi_0$ in \eqref{sum_zero_norm}. We first recall the following definition; see \cite[Page~248]{Rock_wets98}.

\begin{definition}\label{def:equi_lsc}
Let $\{f_{t}\}$ be a sequence of functions. We say that $\{f_{t}\}$ is equi-lsc at $x_0$ if for every $\epsilon >0$ and $\rho > 0$ there exists $\delta >0$ such that
\[
f_{t}(x)\geq \min\{f_{t}(x_0)-\epsilon,\rho\}\text{~for all~}t \text{~and~} \|x-x_0\|\leq \delta.
\]
We say that $\{f_{t}\}$ is equi-lsc on $\mathbb{R}^n$ if $\{f_{t}\}$ is equi-lsc at every point $x_0 \in \mathbb{R}^n$.
\end{definition}

\begin{lemma}\label{equi-lsc}
Let $\{\epsilon_{t}\}$ be a decreasing positive sequence with $\epsilon_{t}\downarrow 0$ and define $\varphi_{\epsilon_{t}}(s) =1-\frac{\log(|s|+\epsilon_{t})}{\log\epsilon_{t}}$. Then the sequence of functions $\{\varphi_{\epsilon_{t}}\}$ is equi-lsc on $\mathbb{R}$.
\end{lemma}
\begin{proof}
We prove by contradiction.
Let $s_0\in \R$. Suppose that $\{\varphi_{\epsilon_{t}}\}$ is not equi-lsc at $s_0$. Then we see from Definition \ref{def:equi_lsc} that there exist $\epsilon_0>0$, $\rho_0 > 0$, a sequence $\{t_j\}$, and a sequence $s_j \rightarrow s_0$ such that
\begin{align}\label{lem:asvar}
 \varphi_{\epsilon_{t_j}}(s_j) < \min\{\varphi_{\epsilon_{t_j}}(s_0) - \epsilon_0,\rho_0\} \text{~ for all ~}j.
\end{align}

If there exists $N$ such that $t_j \le N$ infinitely often, by passing to a further subsequence if necessary, we may assume that $t_j \equiv N_0$ for some integer $N_0$. But this together with \eqref{lem:asvar} contradicts the continuity of $\varphi_{\epsilon_{N_0}}$ at $s_0$.

Thus, we must have $t_j \to \infty$. We then consider the following two cases:
\begin{enumerate}[\rm(a)]
\item Suppose that $s_0\neq 0$. Then clearly $\lim_{j\rightarrow \infty} \varphi_{\epsilon_{t_j}}(s_j) = 1$ and $\lim_{j\rightarrow \infty} \varphi_{\epsilon_{t_j}}(s_0) = 1$. Hence, we have $\lim_{j\rightarrow \infty} \varphi_{\epsilon_{t_j}}(s_j) > \lim_{j\rightarrow \infty} \varphi_{\epsilon_{t_j}}(s_0) - \epsilon_0$. This contradicts \eqref{lem:asvar}.
\item Suppose that $s_0 = 0$. Using the facts that $\epsilon_{t_j}\downarrow 0$ and $s_j\rightarrow s_0 = 0$, we conclude that there exists a positive integer $N$ such that $0<\epsilon_{t_j}\leq \epsilon_{t_j} + |s_j| < 1$ for all $j\geq N$, which further implies $\log(\epsilon_{t_j}) \leq \log(\epsilon_{t_j} + |s_j|)<0$ for all $j\geq N$. Thus, we have $1\ge \frac{\log(\epsilon_{t_j} + |s_j|)}{\log(\epsilon_{t_j})}$ for all $j\geq N$, and hence $\liminf_{j\rightarrow \infty}\varphi_{\epsilon_{t_j}}(s_j)\ge 0$. On the other hand, we have $\varphi_{\epsilon_{t_j}}(s_0) - \epsilon_0 = -\epsilon_0 < 0$ for all $j$. This contradicts \eqref{lem:asvar}.
\end{enumerate}
This completes the proof.
\end{proof}

\section{An algorithmic framework for the MFS$_C$ problem}\label{sec3}

In this section, we introduce an algorithmic framework for solving the MFS$_C$ problem \eqref{MFSC}. Our approach is to solve the equivalent reformulation \eqref{sum_zero_norm}. We construct a sequence of approximation problems with continuous objectives, and solve those approximate problems successively to approximate the original $\ell_0$ optimization problem in \eqref{sum_zero_norm}. A similar approach was previously used in \cite{Zhao12} for solving $\ell_0$ minimization problems to find sparse solutions of linear systems.

Our algorithm, which is an epigraphical approximation scheme for the MFS$_C$ problem (\EAS{}), is presented below as Algorithm~\ref{algorithm_main}.

\begin{algorithm}[h]
\caption{Epigraphical approximation scheme for MFS$_C$ (\EAS{})}\label{algorithm_main}
\begin{algorithmic}
\item[Step 0.] Choose a sequence of continuous functions $\{\Phi_{\epsilon_k}\}$ that is epi-convergent and pointwise convergent to $\Phi_0$. Choose $\tilde x^{0}\in \mathbb{R}^n$. Let $k=1$.
\item[Step 1.]
\begin{enumerate}[{\rm (a)}]
  \item Find an approximate minimizer $\tilde{x}^k$ of $\Phi_{\epsilon_k} + \delta_C$ by an iterative algorithm initialized at $\tilde{x}^{k-1}$.
  \item If a termination criterion is not met, set $k\leftarrow k+1$ and go to Step 1(a).
\end{enumerate}
\end{algorithmic}
\end{algorithm}

In \EAS{}, we first construct a sequence of continuous functions $\{\Phi_{\epsilon_k}\}$. However, different from the literature, we require the sequence of functions to be both {\em epi}-convergent and {\em pointwise} convergent to $\Phi_0$ in \eqref{sum_zero_norm}; see \cite[Chapter~7]{Rock_wets98} for the definition of epi-convergence. Then, in each iteration of our algorithm, we approximately minimize $\Phi_{\epsilon_k}+\delta_C$ and use the approximate minimizer $\tilde{x}^k$ as an initial point for minimizing $\Phi_{\epsilon_{k+1}}+\delta_C$. It can be shown that if $\tilde{x}^k \in \nu_k$-$\argmin(\Phi_{\epsilon_k}+\delta_C)$ with $\nu_k \downarrow 0$, then any accumulation point of $\{\tilde{x}^k\}$ is a minimizer of \eqref{sum_zero_norm}; see \cite[Theorem~7.46(a)]{Rock_wets98} and \cite[Theorem~7.31(b)]{Rock_wets98}.

In order to make use of \EAS{}, we need to specify how to construct the sequence of continuous functions $\{\Phi_{\epsilon_k}\}$ and how to solve the corresponding subproblem. We postpone the discussion of the subproblem to the next section. In the remainder of this section, we discuss how the sequence of continuous functions $\{\Phi_{\epsilon_k}\}$ can be constructed. We start with the following theorem, which suggests a simple way of constructing such a sequence.

\begin{theorem}\label{lem:sure_epi_con}
Let $\{\varphi_{\epsilon_k}(\cdot)\}$ be a sequence of continuous functions on $\R$ that is both epi-convergent and pointwise convergent to $|\cdot|_0$ on $\R$. Define $\Phi_{\epsilon_k}(x) := \sum_{i=1}^m \varphi_{\epsilon_k}(d_{D_i}^2(x))$. Then $\{\Phi_{\epsilon_{k}}\}$ is both epi-convergent and pointwise convergent to $\Phi_{0}$ in \eqref{sum_zero_norm}.
\end{theorem}
\begin{proof}
We start by showing that for each $i$, the sequence of functions $\{\varphi_{\epsilon_k}(d_{D_i}^2(\cdot))\}$ epi-converges to $|d_{D_i}^2(\cdot)|_0$.
In view of \cite[Proposition~7.2]{Rock_wets98}, it suffices to show that, for each $i$ and $x \in \mathbb{R}^n$, it holds that
\begin{equation}\label{epicon}
\left\{
\begin{split}
\liminf_{k\to\infty}\varphi_{\epsilon_k}(d_{D_i}^2({x}^k)) \geq |d_{D_i} ^2 (x)|_0 \text{~ for every sequence ~}{x}^k\rightarrow x,\\
\limsup_{k\to\infty}\varphi_{\epsilon_k}(d_{D_i}^2({x}^k))\leq |d_{D_i} ^2 (x)|_0 \text{~ for some sequence ~}{x}^k\rightarrow x.
\end{split}
\right.
\end{equation}
Since $\{\varphi_{\epsilon_k}(\cdot)\}$ converges pointwise to $|\cdot|_0$, the second relation above holds trivially for the constant sequence ${x}^k \equiv x$. On the other hand, consider any sequence ${x}^k\rightarrow x$. Then we have $d_{D_i}^2({x}^k)\to d_{D_i} ^2 (x)$. Using this together with \cite[Proposition~7.2]{Rock_wets98} and the fact that $\{\varphi_{\epsilon_k}(\cdot)\}$ epi-converges to $|\cdot|_0$, we conclude that the first relation in \eqref{epicon} also holds. Thus, we have shown that $\{\varphi_{\epsilon_k}(d_{D_i}^2(\cdot))\}$ epi-converges to $|d_{D_i}^2(\cdot)|_0$.

Now, notice that $\{\varphi_{\epsilon_k}(d_{D_i}^2(\cdot))\}$ also pointwise converges to $|d_{D_i}^2(\cdot)|_0$ since $\{\varphi_{\epsilon_k}(\cdot)\}$ pointwise converges to $|\cdot|_0$. The desired conclusion now follows from these and \cite[Theorem~7.46]{Rock_wets98}. This completes the proof.
\end{proof}

Based on the above theorem, in order to construct the desired sequence $\{\Phi_{\epsilon_k}\}$ as required in \EAS{}, it suffices to construct a sequence of continuous functions $\{\varphi_{\epsilon_k}(\cdot)\}$ that is both epi-convergent and pointwise convergent to $|\cdot|_0$ on $\R$ and define $\Phi_{\epsilon_k}$ accordingly. We now present some concrete examples of such $\{\varphi_{\epsilon_k}\}$.

\begin{example}\label{lem:3}
Let $\{\epsilon_k\}$ be a decreasing positive sequence with $\epsilon_k\downarrow 0$.
\begin{enumerate}[{\rm (a)}]
  \item Consider $\varphi_{\epsilon}(s)=1-\frac{\log(|s|+\epsilon)}{\log\epsilon}$, which appeared in \cite[Example~2.3(i)]{Zhao12}. We claim that the sequence $\{\varphi_{\epsilon_k}(\cdot)\}$ is both epi-convergent and pointwise convergent to $|\cdot|_0$.

  First, it is routine to show the pointwise convergence. Also, we know from Lemma~\ref{equi-lsc} that the sequence $\{\varphi_{\epsilon_k}\}$ is equi-lsc on $\mathbb{R}$. This together with the pointwise convergence and \cite[Theorem~7.10]{Rock_wets98} shows that $\{\varphi_{\epsilon_k}(\cdot)\}$ also epi-converges to $|\cdot|_0$.

  \item Consider $\varphi_{\epsilon}(s)=\frac{|s|}{|s|+\epsilon} + \epsilon |s|$, which is a modification of the function in \cite[Example~2.6]{Zhao12}. We claim that the sequence $\{\varphi_{\epsilon_k}(\cdot)\}$ is both epi-convergent and pointwise convergent to $|\cdot|_0$.

  Again, it is routine to show the pointwise convergence. Next, define $\varphi_{\epsilon_k,1}(s) := \frac{|s|}{|s|+\epsilon_k}$ and $\varphi_{\epsilon_k,2}(s):= \epsilon_k |s|$. Then the sequence $\{\varphi_{\epsilon_k,1}(\cdot)\}$ is nondecreasing and converges pointwise to $|\cdot|_0$, and the sequence $\{\varphi_{\epsilon_k,2}\}$ is nonincreasing and converges pointwise to $0$. Thus, according to Proposition~7.4(c) and (d) of \cite{Rock_wets98}, we see that $\{\varphi_{\epsilon_k,1}(\cdot)\}$ epi-converges to $|\cdot|_0$, and $\{\varphi_{\epsilon_k,2}\}$ epi-converges to $0$. Since $\varphi_{\epsilon_k} = \varphi_{\epsilon_k,1}+ \varphi_{\epsilon_t,2}$, using the above observations and \cite[Theorem 7.46]{Rock_wets98}, we conclude further that $\{\varphi_{\epsilon}(\cdot)\}$ epiconverges to $|\cdot|_0$.
\end{enumerate}
\end{example}

Suppose a sequence $\{\Phi_{\epsilon_k}\}$ satisfying the requirement of \EAS{} is constructed as described in Theorem~\ref{lem:sure_epi_con}. Then subproblems in the following form have to be approximately solved to obtain $\tilde{x}^k$:
\begin{equation}\label{min_sub}
\min_{x\in C}\ \sum_{i=1}^{m}\varphi_{\epsilon_k} (d_{D_i} ^2 (x)).
\end{equation}
This optimization problem is hard to solve in general. Indeed,
even when $\varphi_{\epsilon_k}$ is chosen to be a smooth function, the objective function in \eqref{min_sub} is still nonsmooth and nonconvex in general if some $D_i$'s are nonconvex. Thus, in the next section, we will restrict our attention to a special class of $\{\varphi_{\epsilon_k}\}$ and discuss how to solve the corresponding problem \eqref{min_sub} efficiently.

\section{A subgradient method for subproblems in \EAS{}}\label{sec4}

In this section, we propose an algorithm for solving the subproblem in Step 1(a) of \EAS{} in the form of \eqref{min_sub} for a large class of choices of $\varphi_\epsilon$.
Specifically, let $\Theta$ denote the set of level-bounded continuous concave functions $\psi:\R_+\to \R_+$ that satisfy the following properties:
\begin{enumerate}
  \item $\psi$ is continuously differentiable on $\mathbb{R}_{++}$ with positive derivative and $\psi(0)=0$;
  \item $\lim_{s\downarrow 0}\psi'(s)$ exists and is positive, and $\psi'_+$ is Lipschitz continuous on $\mathbb{R}_{+}$.
\end{enumerate}
We consider problems of the following form, for a function $\psi\in \Theta$:
\begin{align}\label{sum_subproblem_concave}
\min_{x\in C}\ \Psi (x) := \sum_{i=1}^{m}\psi (d_{D_i} ^2 (x)).
\end{align}
We would like to point out that the assumption $\psi\in \Theta$ in \eqref{sum_subproblem_concave} is general enough to cover the subproblems that arise in Step 1(a) of Algorithm~\ref{algorithm_main} for the two classes of functions studied in Example~\ref{lem:3}: $\varphi_\epsilon(s) = \frac{|s|}{|s|+\epsilon} + \epsilon |s|$, $\epsilon > 0$, and $\varphi_\epsilon(s) = 1-\frac{\log(|s|+\epsilon)}{\log\epsilon}$, $\epsilon\in (0,1)$.
Indeed, for $\varphi_\epsilon(s) = \frac{|s|}{|s|+\epsilon} + \epsilon |s|$, $\epsilon > 0$, the corresponding subproblem \eqref{min_sub} takes the form of \eqref{sum_subproblem_concave} with $\psi(s)= \frac{s}{s+\epsilon} + \epsilon s$; clearly, $\psi\in \Theta$. On the other hand, for $\varphi_\epsilon(s) = 1-\frac{\log(|s|+\epsilon)}{\log\epsilon}$, $\epsilon\in (0,1)$, the subproblem \eqref{min_sub} becomes
\[
\min_{x\in C}\,\sum_{i=1}^{m}\left[1-\frac{\log(d_{D_i} ^2 (x)+\epsilon)}{\log\epsilon}\right].
\]
Since $\epsilon\in (0,1)$ so that $-\log\epsilon > 0$, the above problem is equivalent to
\[
\min_{x\in C}\,\sum_{i=1}^m [\log(d_{D_i} ^2 (x)+\epsilon)-\log\epsilon],
\]
which takes the form of \eqref{sum_subproblem_concave} with $\psi(s) = \log(s+\epsilon) - \log\epsilon$; it is routine to check that $\psi\in \Theta$.

Notice that \eqref{sum_subproblem_concave} is a nonconvex nonsmooth problem in general, and it is not obvious at first glance what algorithm should be applied for solving such a problem. However, in the special case when $D_i$'s are all convex, the functions $x\mapsto d_{D_i} ^2 (x)$, $i=1,\ldots,m$, are smooth, and \eqref{sum_subproblem_concave} can be solved by the classical gradient projection algorithm. This method can be applied efficiently when the projections onto $C$ and $D_i$'s can be computed efficiently, because the gradient of $\Psi$ can be computed in terms of projections onto $D_i$'s:
\begin{align}\label{nablaPsi}
\nabla \Psi(x) = 2\sum_{i=1}^{m}\psi'_+(d_{D_i} ^2 (x))[x-P_{D_i}(x)].
\end{align}
In the general case when $D_i$'s are possibly nonconvex, the function $\Psi$ is not everywhere differentiable in general. Nevertheless, we still have $2\psi'_+(d_{D_i} ^2 (x))[x-\xi] \in \partial (\psi\circ d_{D_i}^2)(x)$ whenever $\xi \in P_{D_i}(x)$,\footnote{This can be proved using the definition of subdifferential, \cite[Example~8.53]{Rock_wets98} and $\psi\in \Theta$.} and the element $2\psi'_+(d_{D_i} ^2 (x))[x-\xi]$ can be computed efficiently if a projection onto $D_i$ can be obtained efficiently. Thus, mimicking the framework of gradient projection algorithm, we propose a subgradient projection algorithm with nonmonotone linesearch for solving \eqref{sum_subproblem_concave}, in which $\nabla \Psi(x)$ is replaced by an element in $\sum_{i=1}^m\partial (\psi\circ d_{D_i}^2)(x)$. Our algorithm, known as subgradient projection algorithm with nonmonotone linesearch (\GP{}), is presented below as Algorithm~\ref{algorithm}. Even though this is a subgradient type algorithm, surprisingly, we can show that the linesearch can be terminated after finitely many inner iterations (i.e., the linesearch is well defined), and that the stepsize sequence $\{\alpha_t\}$ in the algorithm has a uniform lower bound (under an additional assumption on the collection of closed sets; see Theorem~\ref{th:subsequence} below), unlike classical subgradient methods (see, for example, \cite{Bertsekas_99}).

\begin{algorithm}
\caption{Subgradient projection algorithm with nonmonotone linesearch (\GP{}) for \eqref{sum_subproblem_concave}}\label{algorithm}
\begin{algorithmic}
\item[Step 0.] Choose $\alpha_{\max}> \alpha_{\min}>0$, $\eta\in (0,1)$, $\sigma>0$ and an integer $M\geq 0$. Set $t=0$ and pick an $x^0\in C$.
\item[Step 1.]
\begin{enumerate}[{\rm (a)}]
  \item Choose $\alpha^0_t \in [\alpha_{\min},\alpha_{\max}]$ and set $\alpha = \alpha^0_t$. Pick any ${\bm\xi}_i^t\in P_{D_i}(x^t)$ for $i=1,\dots,m$. Set
\begin{align}\label{average_projection}
g^t:= 2\sum_{i=1}^{m}\psi'_+(d^2_{D_i}(x^t))(x^t-{\bm\xi}_i^t).
\end{align}
  \item Choose any
\begin{align}\label{min_u}
\tilde{u} \in \argmin_{u\in C}\left\{ \inner{g^t}{u-x^t} + \frac{1}{2\alpha}\norm{u-x^t}^2 \right\}.
\end{align}
  \item If
  \begin{align}\label{linesearch}
   \Psi(\tilde{u}) - \max\limits_{[t-M]_+ \leq i \leq t} \Psi(x^{i}) \leq - \frac{\sigma}{2}\norm{\tilde{u}-x^t}^2,
\end{align}
go to Step 2. Otherwise, update $\alpha\leftarrow \eta\alpha$ and go to Step 1(b).
\end{enumerate}
\item[Step 2.] Let $\alpha_t := \alpha$, $x^{t+1}: = \tilde{u}$ and go to Step 1.
\end{algorithmic}
\end{algorithm}

We first establish the well-definedness of the linesearch procedure in \GP{}, which is an immediate consequence of the following proposition. For notational convenience, given $x\in C$ and $\bm\xi_i\in P_{D_i}(x)$ for all $i$, for each $\alpha > 0$, let $\tilde u(\alpha)$ denote a minimizer of the problem
\begin{equation}\label{tildeu}
\min_{u\in C} \ \langle g,u - x \rangle + \frac1{2\alpha}\|u-x\|^2,
\end{equation}
where $g := 2\sum_{i=1}^m \psi_+'(d_{D_i} ^2 (x))(x - {\bm\xi}_i)$.

\begin{proposition}[\textbf{Sufficient descent}]\label{lem:decreaseF}
Let $\Psi$ be defined in \eqref{sum_subproblem_concave} with $\psi\in \Theta$, $x\in C$ and $\bm\xi_i\in P_{D_i}(x)$ for all $i$. Let $\alpha > 0$ and $\tilde u(\alpha)$ be defined in \eqref{tildeu}. Then there exists $\beta > 0$ so that
\begin{align}\label{Ff-decrease}
\Psi(\tilde{u}(\alpha)) - \Psi(x)\leq \left(-\frac{1}{\beta} + \sum_{i=1}^{m} \psi'_+(d_{D_i} ^2 (x))\right)\norm{\tilde{u}(\alpha)-x}^2.
\end{align}
Indeed, one can take $\beta = \alpha$ when $C$ is in addition convex and set $\beta = 2\alpha$ otherwise.%
\end{proposition}

\begin{proof}
Since $\psi$ is concave and $\psi'_+$ is continuous on $\mathbb{R}_{+}$, we see that
\begin{align}\label{f1}
\begin{split}
&\Psi(\tilde{u}(\alpha))\leq \sum_{i=1}^{m} \left\{\psi(d^2_{D_i}(x)) + \psi'_+(d^2_{D_i}(x))[d^2_{D_i}(\tilde{u}(\alpha)) - d^2_{D_i}(x)]\right\}\\
&=\Psi(x) + \sum_{i=1}^{m} \psi'_+(d^2_{D_i}(x))[d^2_{D_i}(\tilde{u}(\alpha)) - d^2_{D_i}(x)].
\end{split}
\end{align}
Next, from the definition of distance function, we see that, for each fixed $i$,
\begin{align*}
d^2_{D_i}(x)= \norm{x}^2 - 2h_i(x),
\end{align*}
where $h_i(x):  = \sup\{\inner{x}{y}-\frac{1}{2}\norm{y}^2: y\in D_i\}$. Notice that $h_i$ is finite everywhere and is the pointwise supreme of affine functions. Thus, $h_i$ is a continuous convex function and one can check directly from definition that $P_{D_i}(x)\subseteq \partial h_i(x)$. Then we have
\begin{align}\label{f2}
\begin{split}
&\sum_{i=1}^{m} \psi'_+(d^2_{D_i}(x))[d^2_{D_i}(\tilde{u}(\alpha)) - d^2_{D_i}(x)] \\
&= \sum_{i=1}^{m} \psi'_+(d^2_{D_i}(x)) \{\norm{\tilde{u}(\alpha)}^2 - \|x\|^2 - 2 [h_i(\tilde{u}(\alpha)) - h_i(x)]\}\\
&\leq \sum_{i=1}^{m} \psi'_+(d^2_{D_i}(x)) [\norm{\tilde{u}(\alpha)}^2 - \|x\|^2 -2 \inner{{\bm\xi}_i}{\tilde{u}(\alpha)-x}]\\
&=\inner{g}{\tilde{u}(\alpha)-x} + \sum_{i=1}^{m} \psi'_+(d^2_{D_i}(x))\norm{\tilde{u}(\alpha)-x}^2,
\end{split}
\end{align}
where the first inequality holds because $\psi'_+>0$, $h_i$ is convex and ${\bm\xi}_i\in P_{D_i}(x)\subseteq \partial h_i(x)$ and the last equality holds because of the relation $\norm{\tilde{u}(\alpha)}^2 - \|x\|^2 -2 \inner{{\bm\xi}_i}{\tilde{u}(\alpha)-x} = \|\tilde{u}(\alpha)-x\|^2 + 2\inner{x - {\bm\xi}_i}{\tilde{u}(\alpha)-x}$ and the definition of $g$ in \eqref{tildeu}.

Now, using the definition of $\tilde u(\alpha)$ as a minimizer of \eqref{tildeu} and the fact that $x\in C$, we see that $\inner{g}{\tilde{u}(\alpha) - x} + \frac{1}{2\alpha}\|\tilde{u}(\alpha) - x\|^2 \leq 0$. Combining this with \eqref{f1} and \eqref{f2}, we see that \eqref{Ff-decrease} holds with $\beta = 2\alpha$.

Finally, suppose $C$ is in addition convex. Then the function $f(u):= \inner{g}{u-x} + \frac{1}{2\alpha}\norm{u-x}^2 + \delta_C(u)$ is strongly convex with modulus $\frac{1}{\alpha}$. Using this and the definition of $\tilde u(\alpha)$ as a minimizer of \eqref{tildeu}, we see that $f(x)-f(\tilde{u}(\alpha))\geq \frac{1}{2\alpha}\norm{\tilde{u}(\alpha)-x}^2$. Rearranging terms, we have
\[
\inner{g}{\tilde{u}(\alpha) - x} \leq -\frac{1}{\alpha}\norm{\tilde{u}(\alpha)-x}^2.
\]
This together with \eqref{f1} and \eqref{f2} implies that \eqref{Ff-decrease} holds with $\beta = \alpha$. This completes the proof.
\end{proof}

Using Proposition~\ref{lem:decreaseF}, it is then routine to show the well-definedness of the linesearch procedure in \GP{}.
\begin{corollary}[\textbf{Well-definedness of linesearch}]\label{well-defined}
Let $\Psi$ be defined in \eqref{sum_subproblem_concave} with $\psi\in \Theta$ and suppose that the \GP{} is applied for solving \eqref{sum_subproblem_concave}. Then, in each iteration, the linesearch criterion in Step~1(c) is satisfied after finitely many inner iterations.
\end{corollary}

We next show that the stepsize sequence $\{\alpha_t\}$ generated in \GP{} for solving \eqref{sum_subproblem_concave} has a {\em uniform} lower bound under the additional assumption that $C^{\infty}\cap(\bigcap_{i=1}^m D_i^{\infty})=\{0\}$.\footnote{This assumption is satisfied if $C$ or any $D_i$ is bounded.} We will also show that the sequence $\{x^t\}$ generated by \GP{} is bounded and any accumulation point is a stationary point of the function $\Psi+\delta_C$ in \eqref{sum_subproblem_concave}. Here, we say that $\bar x$ is a stationary point of $\Psi+\delta_C$ if
\begin{align}\label{stationarypoint}
 0\in \sum_{i=1}^{m} [2\psi_+'(d^2_{D_i}(\bar x))(\bar x-P_{D_i}(\bar x))]+N_C(\bar x).
\end{align}
Note that if $\bar x$ is a local minimizer of $\Psi+\delta_C$, then in view of \cite[Theorem~10.1]{Rock_wets98}, \cite[Corollary~10.9]{Rock_wets98}, \cite[Theorem~1.110(ii)]{Morduk06} and \cite[Example~8.53]{Rock_wets98}, one can show that $\bar x$ is a stationary point of $\Psi+\delta_C$.

\begin{theorem}\label{th:subsequence}
Suppose that $C^{\infty}\cap(\bigcap_{i=1}^m D_i^{\infty})=\{0\}$ and let $\Psi$ be defined in \eqref{sum_subproblem_concave} with $\psi\in \Theta$. Let $\{x^t\}$, $\{\alpha_t\}$ and $\{{\bm\xi}_i^t\}$, $i=1,\ldots,m$, be the sequences generated by \GP{}. Then the following statements hold.
\begin{enumerate}[\rm(a)]
  \item The sequences $\{x^t\}$ and $\{{\bm\xi}_i^t\}$, $i=1,\ldots,m$, are all bounded, and $\inf\limits_{t\ge 0}\alpha_t > 0$.
  \item It holds that $\lim\limits_{t\to\infty}\|x^{t+1}-x^t\|= 0$ and the limit $\lim\limits_{t\rightarrow\infty}\Psi(x^{t})$ exists.
  \item {\bf (Global subsequential convergence)} Any cluster point of $\{x^t\}$ is a stationary point of $\Psi+\delta_C$.
\end{enumerate}
\end{theorem}
\begin{proof}
(a) We first note from \eqref{linesearch} and a simple induction argument that $\Psi(x^t)\leq \Psi(x^0)$ for all $t\ge 1$.
This together with the nonnegativity and level-boundedness of $\psi$ implies that the sequence $\{d^2_{D_i}(x^t)\}$ is bounded for each $i=1,\ldots,m$. Hence, there exists $M_1 > 0$ such that
\begin{align}\label{boundxandxi}
\norm{x^t-{\bm\xi}_i^t} = d_{D_i}(x^t)\leq M_1 \text{~for all~} t \text{~and all~} i,
\end{align}
where the first equality holds because ${\bm\xi}_i^t\in P_{D_i}(x^t)$.

Now, we claim that the sequence $\{x^t\}$ is bounded. Suppose to the contrary that $\{x^t\}$ is unbounded. Then there exists a subsequence $\{x^{t_j}\}$ such that $\lim\limits_{j\rightarrow \infty}\norm{x^{t_j}}= + \infty$. By passing to a further subsequence if necessary, we may assume without loss of generality that $\|x^{t_j}\|\neq 0$ for all $j$ and $\lim_{j\to\infty}\frac{x^{t_j}}{\|x^{t_j}\|} = d$ for some $d$. Then $\|d\|=1$ and we also have from the definition of horizon cone that $d\in C^\infty$. Next, dividing $\|x^{t_j}\|$ from both sides of \eqref{boundxandxi}, we see that
\begin{align*}
0\leq \left\|\frac{{\bm\xi}^{t_j}_i}{\norm{x^{t_j}}}-\frac{x^{t_j}}{\norm{x^{t_j}}}\right\| = \frac{\norm{{\bm\xi}^{t_j}_i-x^{t_j}}}{\norm{x^{t_j}}}\leq \frac {M_1}{\norm{x^{t_j}}}.
\end{align*}
Passing to the limit in the above inequality, we see that $\lim_{k\to\infty}\frac{{\bm\xi}^{t_j}_i}{\norm{x^{t_j}}} = d$ for each $i$. Using this together with the definition of horizon cone of $D_i^{\infty}$ and the fact ${\bm\xi}_i^{t_j}\in D_i$ for all $i$ and $j$, we conclude further that $d \in \bigcap_{i=1}^m D_i^{\infty}$. Since we also have $\|d\|=1$ and $d\in C^\infty$, this contradicts the assumption that $C^{\infty}\cap\bigcap_{i=1}^m D_i^{\infty} =\{0\}$. Thus, the sequence $\{x^t\}$ is bounded. In view of \eqref{boundxandxi}, we see that $\{{\bm\xi}_i^t\}$ is also bounded for all $i$.

Next, we show that $\inf_{t\ge 0}\alpha_t > 0$. We first note from \eqref{boundxandxi} and the positivity and continuity of $\psi_+'$ on $\R_+$ that there exists $M_2 > 0$ so that $0\leq \psi_+'(d^2_{D_i}(x^t))\leq M_2$ for all $t$ and $i$. Using this fact and applying Proposition~\ref{lem:decreaseF} with $x = x^t$ and $\bm\xi_i = \bm\xi_i^t$, we see that
\[
\begin{aligned}
\Psi(\tilde{u}(\alpha)) - \Psi(x^t) \leq \left(-\frac1{2\alpha} + mM_2\right)\norm{\tilde{u}(\alpha)-x^t}^2,
\end{aligned}
\]
where $\tilde u(\alpha)$ is defined as in \eqref{tildeu} with $x = x^t$ and $g = g^t$ defined in \eqref{average_projection}, i.e., $\tilde u(\alpha)$ is a minimizer of \eqref{min_u} in the $t$-th iteration. Thus, the linesearch criterion \eqref{linesearch} is satisfied if $\alpha \le (2mM_2 + \sigma)^{-1}$. Hence, using the definition of $\alpha_t$, we must have either $\alpha_t = \alpha^0_t\ge \alpha_{\min}$ (if $\alpha^0_t\le (2mM_2 + \sigma)^{-1}$) or $\alpha_t > \eta(2mM_2 + \sigma)^{-1}$. Consequently, it holds that $\inf_{t\ge 0}\alpha_t \ge \min\{\alpha_{\min},\eta(2mM_2 + \sigma)^{-1}\} > 0$.

(b) This can be proved similarly as in \cite[Lemma~4]{Wright_nowak09}.

(c) Let $x^*$ be a cluster point of the bounded sequence $\{x^t\}$ and let $\{x^{t_j}\}$ be a convergent subsequence with limit $x^*$. Since $\{{\bm\xi}_i^t\}$ is also bounded for all $i$, by passing to a further subsequence if necessary, we may assume without loss of generality that, for each $i = 1,\dots,m$, $\bm\xi_i^{t_j}\rightarrow \bm\xi^*_i$ for some $\bm\xi^*_i$. Next, using the definition of $x^{t_j+1}$ as a minimizer of \eqref{min_u} when $\alpha = \alpha_{t_j}$, we see that
\begin{align}\label{fis order1}
0\in 2\sum_{i=1}^{m}\psi'_+(d_{D_i}^2(x^{t_j}))(x^{t_j}-{\bm\xi}_i^{t_j}) + \frac{1}{\alpha_{t_j}}(x^{{t_j}+1} - x^{t_j}) + N_C(x^{{t_j}+1}).
\end{align}
Notice that $\psi_+'(d_{D_i}^2 (\cdot))$ is continuous, $\lim\limits_{t\rightarrow \infty}\|x^{t+1}-x^t\|=0$ according to (b) and $\{\frac{1}{\alpha_t}\}$ is bounded according to (a). Then, passing to the limit as $j\rightarrow \infty$ in \eqref{fis order1}, we have
\begin{align*}
0\in \sum_{i=1}^{m}[2\psi_+'(d_{D_i}^2(x^*))(x^*-{\bm\xi}^*_i)]+ N_C(x^*).
\end{align*}
To complete the proof, it now remains to show that, for each $i$, ${\bm\xi}^*_i\in P_{D_i}(x^*)$. To this end, we first note that ${\bm\xi}^*_i\in D_i$ because $D_i$ is closed and $\{\bm\xi_i^{t_j}\}\subseteq D_i$. In addition, we we have $d_{D_i}(x^*) = \norm{x^*-{\bm\xi}^*_i}$ because $d_{D_i}(x^{t_j})=\|x^{t_j} - {\bm\xi}_i^{t_j}\|$ (thanks to $\bm\xi_i^{t_j}\in P_{D_i}(x^{t_j})$). Thus, we conclude that $\bm\xi^*_i\in P_{D_i}(x^*)$ and this completes the proof.
\end{proof}

Before ending this section, as a little digression and an immediate application of Theorem~\ref{th:subsequence}, we discuss global convergence of the averaged projection algorithm.
Averaged projection algorithm is a popular algorithm for finding a point of intersection of several nonempty closed sets $D_1,\ldots,D_m$. In this algorithm, we initialize at an $x^0\in \R^n$ and update
\begin{equation}\label{eq_avg_proj}
x^{t+1}\in \frac1m\sum_{i=1}^m P_{D_i}(x^t).
\end{equation}
When each $D_i$ is convex, the above algorithm is just the proximal gradient algorithm applied to $\frac1{2m}\sum_{i=1}^m d_{D_i}^2(x)$ with constant stepsize $1$, and its convergence is well known. However, the global convergence of the above algorithm is unknown if $D_i$'s are nonconvex: only local convergence was proved recently in \cite{Lewis-Luke09} under suitable regularity assumptions.

We next show that the averaged projection algorithm \eqref{eq_avg_proj} is a special case of \GP{} when $C = \R^n$ and $\psi(s) = \frac{s}{m}$, which clearly belongs to $\Theta$. Hence, we obtain as an immediate corollary of Theorem~\ref{th:subsequence} that the averaged projection algorithm \eqref{eq_avg_proj} is globally subsequentially convergent when $\bigcap_{i=1}^m D_i^\infty=\{0\}$.

\begin{proposition}
Suppose that $C=\mathbb{R}^n$, $\psi(s)= \frac{s}{m}$ and let $\Psi$ be defined in \eqref{sum_subproblem_concave}. Let $\alpha^0_t = \frac{1}{2}$ for all $t$ and $0<\sigma \leq 2$ in \GP{}. Then the \GP{} reduces to the averaged projection algorithm.
\end{proposition}
\begin{proof}
It suffices to show that $x^{t+1}= \frac{1}{m}\sum_{i=1}^m{\bm\xi}_i^t$ in every iteration of \GP{} under the assumptions.
To this end, we first apply Proposition~\ref{lem:decreaseF} with $x = x^t$ and $\bm\xi_i = \bm\xi_i^t$, and invoke $\psi'_+\equiv \frac1m$ to obtain
\[
\begin{aligned}
 \Psi(\tilde{u}(0.5)) - \Psi(x^t) \leq (-2+1)\norm{\tilde{u}(0.5)-x^t}^2 \leq - \frac{\sigma}{2}\norm{\tilde{u}(0.5)-x^t}^2,
\end{aligned}
\]
where $\tilde u(\alpha)$ is defined as in \eqref{tildeu} with $x = x^t$ and $g = g^t$ defined in \eqref{average_projection} (i.e., $\tilde u(\alpha)$ is a minimizer of \eqref{min_u} in the $t$-th iteration.), and the last inequality holds because $0<\sigma\leq 2$.
This implies that $\tilde{u}(0.5)$ satisfies \eqref{linesearch} and hence $x^{t+1}= \tilde{u}(0.5)$. Using this, the first-order optimality condition of the subproblem \eqref{min_u} with $\alpha = 0.5$ and the fact that $C=\mathbb{R}^n$, we see further that
$0= g^t + 2 (x^{t+1} - x^t)$. Thus
\begin{align*}
x^{t+1} = x^t - \frac{1}{2} g^t = x^t - \left(\frac{1}{m}\sum_{i=1}^{m}[x^t-{\bm\xi}_i^t]\right) = \frac{1}{m}\sum_{i=1}^{m}{\bm\xi}_i^t,
\end{align*}
where the second equality follows from \eqref{average_projection} and the fact that $\psi_+'\equiv \frac{1}{m}$ on $\R_+$. This completes the proof.
\end{proof}

\begin{corollary}\label{coro:aver}
  Suppose that $\bigcap_{i=1}^m D_i^\infty = \{0\}$ and let $\{x^t\}$ be the sequence generated by the averaged projection algorithm \eqref{eq_avg_proj}. Then the sequence $\{x^t\}$ is bounded, $\lim_{t\to\infty}\|x^{t+1}-x^t\| = 0$ and any cluster point of $\{x^t\}$ is a stationary point of $\Psi+\delta_C$ in \eqref{sum_subproblem_concave} with $\psi(s) = \frac{s}m$ and $C = \R^n$.
\end{corollary}

\section{Global sequential convergence of \GP{} with $M = 0$}\label{sec5}

In this section, we study convergence of the whole sequence generated by \GP{} with $M = 0$ for solving \eqref{sum_subproblem_concave}. We consider two cases in Sections~\ref{sec5.1} and \ref{sec5.2}, respectively: (1) each $D_i$ is convex; (2) some $D_i$'s are possibly nonconvex and $C = \R^n$. We establish global convergence of the whole sequence generated by \GP{} with $M = 0$ in these two cases by assuming KL properties of suitable potential functions. Then, in Section~\ref{sec5.3}, we discuss a relationship between the KL properties used in Section~\ref{sec5.1} and Section~\ref{sec5.2}.

\subsection{Global sequential convergence of \GP{} with $M = 0$ for convex $D_i$}\label{sec5.1}
In this subsection, we assume that each $D_i$ is convex, but $C$ can still be possibly nonconvex.
We show in the next theorem the global convergence of the whole sequence generated by \GP{} with $M = 0$ under the assumption that $\Psi+\delta_C$ is a KL function.

\begin{theorem}\label{thm5.1}
Suppose that each $D_i$ is convex, $C^{\infty} \cap \bigcap_{i=1}^m D_i^{\infty}=\{0\}$ and $\Psi + \delta_C$ is a KL function, where $\Psi$ is defined in \eqref{sum_subproblem_concave} with $\psi\in \Theta$. Let $\{x^t\}$ be the sequence generated by \GP{} with $M = 0$. Then $\{x^t\}$ is globally convergent to a stationary point of $\Psi + \delta_C$.
\end{theorem}
\begin{proof}
In view of Theorem~\ref{th:subsequence}, it suffices to show that the sequence $\{x^t\}$ is convergent. Since ${\rm dom}\partial(\Psi + \delta_C) = C$ thanks to the smoothness of $\Psi$, and $\Psi + \delta_C$ is a KL function by assumption, according to \cite[Theorem 2.9]{Attouch-Bolte13}, we only need to check that $\{x^t\}$ satisfies the conditions {\bf H1}, {\bf H2} and {\bf H3} there.

\textbf{H1}: Since $M=0$ and $\{x^t\}\subseteq C$, we see from \eqref{linesearch} that $\Psi(x^{t+1}) + \frac{\sigma}{2}\norm{x^{t+1} -x^t}^2 \leq \Psi(x^{t})$.

\textbf{H2}: We need to check that for each $t$ there exists $\omega^{t+1}\in \partial (\Psi+\delta_C)(x^{t+1})$ so that $\|\omega^{t+1}\|\leq b \norm{x^{t+1} - x^t}$ for some $b > 0$ independent of $t$.

To this end, we note first that $x^{t+1}$ is a minimizer of \eqref{min_u} when $\alpha = \alpha_{t}$. Hence, we have $-\nabla\Psi(x^t) - \frac{1}{\alpha_t}(x^{t+1} - x^t) \in N_C(x^{t+1})$. Now, define
\[
\omega^{t+1} := \nabla\Psi(x^{t+1})-\nabla\Psi(x^t) - \frac{1}{\alpha_t}(x^{t+1} - x^t).
\]
Then we have $\omega^{t+1}\in \nabla \Psi(x^{t+1}) + N_C(x^{t+1}) = \partial(\Psi+\delta_C)(x^{t+1})$, where the equality follows from \cite[Exercise\,8.8]{Rock_wets98}. On the other hand, note from \eqref{nablaPsi} that $\nabla \Psi$ is locally Lipschitz because $\psi'_+$ is Lipschitz.
Also, recall from Theorem \ref{th:subsequence}(a) that $\{x^t\}$ is bounded. Using these and the definition of $w^{t+1}$, we conclude that
\[
\|w^{t+1}\| = \left\|\nabla\Psi(x^{t+1})-\nabla\Psi(x^t) - \frac{1}{\alpha_t}(x^{t+1} - x^t)\right\| \le \left(c + \frac1{\alpha_0}\right)\|x^{t+1} - x^t\|,
\]
where $c$ is the Lipschitz continuity modulus of $\nabla \Psi$ on a compact set containing $\{x^t\}$, and $\alpha_0 := \inf_{t\geq 0}\, \alpha_t$, which is positive thanks to Theorem \ref{th:subsequence}(a). Thus, {\bf H2} holds.

\textbf{H3}: This follows from the boundedness of $\{x^t\}$ by Theorem \ref{th:subsequence}(a), the continuity of $\Psi$ and the closedness of $C$.
\end{proof}

Based on the assumptions of Theorem~\ref{thm5.1}, it is now routine (see \cite[Theorem~3.4]{AtBoReSo10} for a similar analysis) to establish the local convergence rate of the sequence generated by \GP{} with $M = 0$ under the additional assumption that $\Psi+\delta_C$ is a KL function with exponent $\theta\in [0,1)$. In particular, an exponent of $\theta=\frac12$ implies local linear convergence of the sequence generated.
Thus, in the next theorem, we discuss conditions on $\{C,D_1,\dots,D_m\}$ that will guarantee $\Psi + \delta_C$ to be a KL function with exponent $\frac12$. Specifically, we assume that
$\{C,D_1,\dots,D_m\}$ is a collection of closed convex sets that is boundedly linearly regular. Recall that a collection of closed convex sets $\{C,D_1,\dots,D_m\}$ is boundedly linearly regular if $C\cap \bigcap_{i=1}^m D_i\neq \emptyset$ and for every bounded set $B$, there exists $c > 0$ so that
\[
\dist \left(x,C\cap \bigcap_{i=1}^m D_i\right) \le c  \max\{ d_C(x),\max_{1\le i\le m} d_{D_i}(x)\}
\]
whenever $x\in B$. It is known that when $C$ and $D_i$ are polyhedra with nonempty intersection, then $\{C,D_1,\dots,D_m\}$ is boundedly linearly regular; see \cite[Theorem 5.6.2]{Bauschke_Borwein96}.

\begin{theorem}\label{kl_exponent_convex}
Suppose that $C$ and all $D_i$'s are convex and let $\Psi$ be defined in \eqref{sum_subproblem_concave} with $\psi\in \Theta$. If $\{C,D_1, \cdots, D_m\}$ is boundedly linearly regular, then $\Psi+\delta_C$ is a KL function with exponent $\frac{1}{2}$.
\end{theorem}
\begin{proof}
For notational simplicity, we write $D= \bigcap_{i=1}^m\,D_i$. We first show that the set of stationary points of $\Psi +\delta_C$, denoted by $\cal X$, is $C\cap D$. Note that $C\cap D \subseteq \cal X$ can be shown by a direct verification using the definition of stationary points in \eqref{stationarypoint}. Conversely, suppose that $\bar x\in \cal X$, i.e., it satisfies \eqref{stationarypoint}. Note that $d^2_{D_i}$ is convex due to the convexity of $D_i$, $C$ is convex by assumption and $\psi_+' > 0$ on $\R_+$. Hence, the function $x\mapsto \sum_{i=1}^{m} [\psi_+'(d_{D_i}^2(\bar x))d^2_{D_i}(x)] + \delta_C(x)$ is a convex function. Using this, \eqref{stationarypoint} and $\nabla d^2_{D_i}(x) = 2 (x - P_{D_i}(x))$, we see further that
\begin{align*}
\bar x\in \argmin_{x\in C}\sum_{i=1}^{m} \psi_+'(d_{D_i}^2(\bar x))d^2_{D_i}(x) = C\cap D,
\end{align*}
where the equality holds because $\psi_+' > 0$ on $\R_+$ and $C\cap D \neq\emptyset$. Thus, we have shown that $C\cap D = \cal X$.

Since $\mathcal{X} =C\cap D = \argmin(\Psi + \delta_C)$, in view of \cite[Lemma 2.1]{LiPong17} and Definition \ref{KL-exponenct}, we only need to check that, for any fixed $\bar{x}\in C\cap D$, there exist positive numbers $c$ and $r$ so that
\[
\dist(0,\partial (\Psi + \delta_C)(x)) \geq c (\Psi(x) - \Psi(\bar{x}))^{\frac{1}{2}}\text{~ whenever ~} x\in C \cap \mathbf{B}(\bar{x},r).
\]

To this end, fix any $\bar x\in C\cap D$ and any $r>0$. 
Since $\psi'_+(d^2_{D_i}(\cdot))$ is continuous on $\mathbf{B}(\bar{x},r)$ and $\psi'_+ >0$ on $\R_+$, we see that there exist $\nu_1$ and $\nu_2$ such that $0 < \nu_1 \leq \psi'_+(d^2_{D_i}(x)) \leq \nu_2$ for all $x \in \mathbf{B}(\bar{x},r)$ and all $i=1,\ldots,m$. Now, define $\omega_i(x) := \frac{\psi'_+(d^2_{D_i}(x))}{\sum_{j=1}^m\, \psi'_+(d^2_{D_j}(x))}$.
Then we have for all $x \in \mathbf{B}(\bar{x},r)$ that
\begin{equation}\label{w_property}
\sum_{i=1}^m\, \omega_i(x) =1\ \ \text{ and }\ \ \min_{1\le i\le m}\omega_i(x) \geq \nu := \frac{\nu_1}{m \nu_2} >0.
\end{equation}

Let $x \in C\cap \mathbf{B}(\bar{x},r)$ and $y \in C \cap D$. Since $y\in D_i$ for each $i$, it follows that
\begin{align}\label{ineq:4}
\begin{split}
\langle P_{D_i}(x) -x , y -x \rangle &= \langle P_{D_i}(x) -x , y - P_{D_i}(x) + P_{D_i}(x) - x \rangle \geq  d^2_{D_i}(x),
\end{split}
\end{align}
where the inequality holds because $\langle P_{D_i}(x) -x , y - P_{D_i}(x)\rangle\ge 0$, which is a consequence of the convexity of $D_i$ and the definition of $P_{D_i}(x)$. Now, for any $\zeta \in N_C(x)$, we have
\begin{align*}
&\sum_{i=1}^m\,\omega_i(x) d^2_{D_i}(x) \leq  \left\langle \sum_{i=1}^m\,\omega_i(x) P_{D_i}(x) - x , y -x \right\rangle \\
&= \left\langle \sum_{i=1}^m\,\omega_i(x) P_{D_i}(x) - x - \zeta, y -x \right\rangle + \left\langle \zeta, y -x \right\rangle \\
& \leq  \left\| \sum_{i=1}^m\,\omega_i(x) P_{D_i}(x) - x - \zeta\right\| \| y -x \|,
\end{align*}
where the first inequality is due to \eqref{w_property} and \eqref{ineq:4}, the second inequality holds because $\zeta \in N_C(x)$ and $y\in C$. Taking infimum over $y \in C\cap D$ and $\zeta\in N_C(x)$ in the above inequality and noting that
\[
\partial(\Psi+\delta_C)(x) = 2 \left(\sum_{j=1}^{m}\psi'_+ (d^2_{D_j}(x))\right) \left[x - \sum_{i=1}^{m}\omega_i(x) P_{D_i}(x)+ N_C(x)\right],
\]
we obtain further that
\[
\sum_{i=1}^m\,\omega_i(x) d^2_{D_i}(x) \leq c_1\dist (0,\partial(\Psi+\delta_C)(x)) \cdot d_{C\cap D}(x),
\]
where $c_1 = (2m \nu_1)^{-1}$. Using this together with the bounded linear regularity of $\{C,D_1, \cdots, D_m\}$, we see further that for any $x\in C\cap \mathbf{B}(\bar{x},r)$, we have
\[
\begin{aligned}
&\nu \max_{1\leq j \leq m}\{d^2_{D_j}(x)\}\le  \sum_{i=1}^m\,\omega_i(x) d^2_{D_i}(x)\leq c_1 \dist (0,\partial(\Psi+\delta_C)(x)) \cdot d_{C\cap D} (x)\\
&\le c_2\dist (0,\partial(\Psi+\delta_C)(x))\max_{1\leq j \leq m}\{d_{D_j}(x)\}
\end{aligned}
\]
for some constant $c_2 > 0$, where the first inequality follows from \eqref{w_property}, and the last inequality holds because $\{C,D_1, \cdots, D_m\}$ is boundedly linearly regular and $x\in C$.
Thus, we have for all $x\in C\cap \mathbf{B}(\bar{x},r)$ that
\begin{align}\label{ineq:5}
\max_{1\leq j \leq m}\{d_{D_j}(x)\}\le \frac{c_2}{\nu} \dist (0,\partial (\Psi + \delta_C)(x)).
\end{align}

On the other hand, note that for any $x \in C\cap \mathbf{B}(\bar{x},r)$, we have $P_{C\cap D}(x) \in \mathbf{B}(\bar{x},r)$ since the projection operator is nonexpansive and $\bar x\in C\cap D$. Moreover, note from \eqref{nablaPsi} that $\nabla \Psi(x)$ is locally Lipschitz because $\psi'_+$ is Lipschitz. Thus, we deduce further that for any $x \in C\cap \mathbf{B}(\bar{x},r)$,
\begin{align*}
&\Psi(x) - \Psi(\bar{x}) = \Psi(x) - \Psi(P_{C\cap D}(x))\\
&\leq \langle \nabla\,\Psi(P_{C\cap D}(x)), x - P_{C\cap D}(x) \rangle + \frac{L_1}{2} \| x - P_{C\cap D}(x) \|^2\\
&= \frac{L_1}{2} d_{C\cap D}^2(x)\le \frac{L_1}{2}c_3 \max_{1\leq j \leq m}\{d^2_{D_j}(x)\} \leq c_4 \dist^2(0,\partial(\Psi + \delta_C)(x)),
\end{align*}
where $L_1$ is the Lipschitz continuity modulus of $\nabla \Psi$ on $\mathbf{B}(\bar{x},r)$, the first equality holds because $\Psi(\bar{x}) = \Psi(P_{C\cap D}(x)) = 0$, the second equality holds because $\nabla\,\Psi(P_{C\cap D}(x)) = 0$ by direct computation, the existence of $c_3$ follows from the bounded linear regularity of $\{C,D_1, \cdots, D_m\}$ and $x\in C$, and the existence of $c_4$ follows from \eqref{ineq:5}. This completes the proof.
\end{proof}

\subsection{Global sequential convergence of \GP{} with $M = 0$ for nonconvex $D_i$}\label{sec5.2}

In this subsection, we assume that $C=\R^n$ but allow each $D_i$ to be possibly nonconvex:
note that the function $d^2_{D_i}$ is not smooth when $D_i$ is nonconvex. We will study global convergence of the whole sequence generated by \GP{} with $M = 0$ for solving \eqref{sum_subproblem_concave} in this case.

Our analysis below will be based on the following potential function:
\begin{align}\label{def:F}
\F(x,{\bm\xi},u):=  \sum_{i=1}^{m}\left[ \rho^* (u_i) - u_i\norm{x- \bm\xi_i}^2 + \delta_{D_i}(\bm\xi_i) \right],
\end{align}
where $\bm\xi=(\bm\xi_1,\cdots,\bm\xi_m)\in \R^{mn}$ with each $\bm\xi_i \in \mathbb{R}^n$, $u\in \R^m$, $\rho$ is the convex function defined by
\begin{equation}\label{liftingrho}
\rho(s):=
\begin{cases}
-\psi(s) & {\rm if}\ s \geq 0,\\
s^2-l s  & {\rm otherwise},
\end{cases}
\end{equation}
with $\psi\in \Theta$ (see \eqref{sum_subproblem_concave}) and $l := \psi'_+(0)>0$, and $\rho^*$ is the convex conjugate. One can show that
\begin{equation}\label{rho'}
\rho'(s)=
\begin{cases}
-\psi_+'(s)& {\rm if}\ s \geq 0,\\
 2s-l     & {\rm otherwise}.
\end{cases}
\end{equation}
We then see immediately from \eqref{liftingrho}, \eqref{rho'} and $\psi\in \Theta$ that $\rho$ is a continuously differentiable nonincreasing convex function on $\mathbb{R}$ and $\rho'$ is Lipschitz continuous on $\mathbb{R}$.

We collect some essential properties of the potential function $\F$ that will be useful in our subsequent analysis. First, fix any $x\in \R^n$. For each $i$, pick any $\widehat {\bm \xi_i}\in P_{D_i}(x)$ and define $\widehat u_i := \rho'(d^2_{D_i}(x))$. Then we see from \eqref{rho'} that
\begin{align}\label{def:u}
\widehat u_i = \rho'(d^2_{D_i}(x)) = -\psi_+'(d^2_{D_i}(x)) < 0.
\end{align}
From this, we deduce further that
\begin{equation}\label{eq:psiand F}
  \F(x,\widehat{\bm \xi},\widehat u) = \sum_{i=1}^{m}\left[ \rho^* (\widehat u_i) - \widehat u_id^2_{D_i}(x) \right] = \sum_{i=1}^{m}\left[ -\rho(d^2_{D_i}(x))\right]= \Psi(x)\geq 0,
\end{equation}
where the first equality holds because $\widehat {\bm \xi_i}\in P_{D_i}(x)$, while the second equality follows from \eqref{def:u} and \cite[Proposition~11.3]{Rock_wets98}.
Finally, using \cite[Exercise~8.8]{Rock_wets98} and \cite[Proposition~10.5]{Rock_wets98}, we have the following formula for the subdifferential of $\F$ at any $(x,\bm\xi,u) \in \R^n\times \R^{mn}\times \R^n$:
\begin{equation}\label{partial:F}
\partial \F(x,\bm\xi,u)= \left[
\begin{split}
&-2\sum_{i=1}^{m} u_i(x- \bm\xi_i)\\
&\left[ - 2u_i (\bm\xi_i - x) + N_{D_i}(\bm\xi_i) \right]_{i=1}^m \\
&\left[ \partial \rho^*(u_i) - \|x- \bm\xi_i\|^2\right]_{i=1}^m
\end{split}
\right].
\end{equation}

We next bound the distance from the origin to $\partial \F(x,\bm\xi,u)$ along a certain sequence.
\begin{lemma}
Suppose that $\bigcap_{i=1}^m D_i^{\infty}=\{0\}$, $C = \R^n$ and $\psi \in \Theta$. Let $\F$ be defined in \eqref{def:F} and let $\{x^t\}$ and $\{{\bm\xi}_i^t\}$, $i=1,\ldots,m$, be the sequences generated by \GP{}. Define, for each $t$, $u^t := [\rho'(d_{D_i}^2(x^t))]_{i=1}^m$. Then there exists $c >0$ such that for all $t\ge 0$, we have
\begin{align}\label{parti:bound}
\dist (0,\partial \F(x^{t},{\bm\xi}^t,u^t))&\leq c \norm{x^{t+1}-x^{t}}.
\end{align}
\end{lemma}
\begin{proof}
Since $C = \R^n$ and $x^{t+1}$ is a minimizer of \eqref{min_u} when $\alpha = \alpha_{t}$, we have, using the definition of $u^t$ and the expression of $\rho'$ in \eqref{rho'}, that
\[
0 = g^t + \frac{1}{\alpha_{t}}(x^{t+1} - x^{t})=  -2\sum_{i=1}^{m}u_i^{t}(x^{t}-{\bm\xi}_i^{t}) + \frac{1}{\alpha_{t}}(x^{t+1} - x^{t}).
\]
Combining this with \eqref{partial:F}, we deduce that
\begin{align*}
\partial_{x} \F(x^{t},{\bm\xi}^t,u^t)= -2 \sum_{i=1}^{m} u_i^t(x^t- {\bm\xi}_i^t) = - \frac{1}{\alpha_{t}}(x^{t+1} - x^{t}).
\end{align*}
Since $\underline{\alpha} := \inf_{t\ge 0}\alpha_t >0$ according to Theorem \ref{th:subsequence}(a), we obtain further that
\begin{align}\label{convergence_x1}
\dist (0,\partial_{x} \F(x^{t},{\bm\xi}^t,u^t))&\leq {\underline{\alpha}}^{-1}\norm{x^{t+1} - x^{t}}.
\end{align}

Next, recall that ${\bm\xi}_i^t \in P_{D_i}(x^t)$, which implies $x^t - {\bm\xi}_i^t \in N_{D_i}({\bm\xi}_i^t)$, thanks to \cite[Example~6.16]{Rock_wets98}. Moreover, using the expression of $\rho'$ in \eqref{rho'} and the definition of $u^t$ together with the assumption that $\psi\in \Theta$, we have $-u_i^t = -\rho'(d_{D_i}^2(x^t)) = \psi_+'(d^2_{D_i}(x^t)) > 0$. Hence, $-2u_i^t (x^t - {\bm\xi}_i^t) \in N_{D_i}({\bm\xi}_i^t)$. This together with \eqref{partial:F} gives $0\in \partial_{{\bm\xi}}\F(x^t,{\bm\xi}^t,u^t)$. Thus,
\begin{align}\label{convergence_xi}
\dist (0,\partial_{{\bm\xi}} \F(x^{t},{\bm\xi}^t,u^t)) = 0.
\end{align}

Finally, since $u_i^t = \rho'(d_{D_i}^2(x^t))$ implies $d^2_{D_i}(x^t)\in\partial \rho^*(u_i^t)$ according to \cite[Proposition~11.3]{Rock_wets98} and note that ${\bm\xi}_i^t \in P_{D_i}(x^t)$, we see that $\|x^t - {\bm\xi}_i^t\|^2\in \partial \rho^*(u_i^t)$. This together with \eqref{partial:F} gives $0\in \partial_u \F(x^t,{\bm\xi}^t,u^t)$. The desired bound \eqref{parti:bound} now follows immediately from this, \eqref{convergence_x1} and \eqref{convergence_xi}.
\end{proof}

We are now ready to present our convergence analysis.

\begin{theorem}\label{KL_nonconvex}
Suppose that $\bigcap_{i=1}^m D_i^{\infty}=\{0\}$, $C = \R^n$, $\psi \in \Theta$ and $\F$ in \eqref{def:F} is a KL function. Let $\{x^t\}$ be the sequence generated by \GP{} with $M = 0$. Then $\{x^t\}$ is globally convergent to a stationary point of $\Psi$.
\end{theorem}
\begin{proof}
In view of Theorem~\ref{th:subsequence}, it suffices to show that the sequence $\{x^t\}$ is convergent.
We first note that $\max_{[t-M]_+ \leq i \leq t} \Psi(x^{i}) \equiv \Psi(x^{t})$ since $M=0$. Using this and \eqref{linesearch}, we have
\begin{align}\label{decreas:F}
\F(x^{t+1},{\bm\xi}^{t+1},u^{t+1}) - \F(x^{t},{\bm\xi}^{t},u^{t}) = \Psi(x^{t+1}) -\Psi(x^t) \leq - \frac{\sigma}{2}\norm{x^{t+1} - x^t}^2,
\end{align}
where $\{{\bm\xi}^t\}$ is generated by \GP{} and $u^t := [\rho'(d_{D_i}^2(x^t))]_{i=1}^m$, and the first equality follows from \eqref{eq:psiand F} and the definitions of $\widehat{\bm\xi}$, ${\bm\xi}^t$, $\widehat u$ and $u^t$.

Next, we note from Theorem \ref{th:subsequence}(a) and the definition of $u^t$ that the sequence $\{(x^t,{\bm\xi}^t,u^t)\}$ is bounded.
Let $\Omega$ be the set of cluster points of $\{(x^t,{\bm\xi}^t,u^t)\}$. Then $\Omega$ is nonempty and compact. We now show that $\F$ is constant on $\Omega$. To this end, we first observe from \eqref{eq:psiand F} that $\F(x^t,{\bm\xi}^t,u^t) = \Psi(x^t)$ for all $t$. This together with Theorem \ref{th:subsequence}(b) implies that $\lim_{t\rightarrow \infty}\F(x^t,{\bm\xi}^t,u^t) = l_*$ for some $l_*$. Now, take any $(x^*,{\bm\xi}^*,u^*)\in \Omega$. Then there exists a convergent subsequence $\{(x^{t_j},{\bm\xi}^{t_j},u^{t_j})\}$ converging to it. Using \eqref{decreas:F}, we have
\[
\begin{aligned}
  &\F(x^{t_j},{\bm\xi}^{t_j},u^{t_j}) \le \F(x^{t_j-1},{\bm\xi}^{t_j-1},u^{t_j-1}) - \frac\sigma{2}\|x^{t_j} - x^{t_j-1}\|^2\\
  &\le \F(x^{t_j-1},{\bm\xi}^{t_j-1},u^*) - \frac\sigma{2}\|x^{t_j} - x^{t_j-1}\|^2\le \F(x^{t_j-1},{\bm\xi}^*,u^*) - \frac\sigma{2}\|x^{t_j} - x^{t_j-1}\|^2,
\end{aligned}
\]
where the second inequality follows from the definition of $\{u^t\}$ so that $u^t$ is a minimizer of $u\mapsto \F(x^t,{\bm\xi}^t,u)$, while the last inequality follows from the facts that ${\bm\xi}^{t_j-1}_i\in P_{D_i}(x^{t_j-1})$ and that $u^t_i < 0$ (so that $u^*_i\le 0$) for all $i$. Passing to the limit in the above inequality and invoking the definition of $l_*$ and Theorem~\ref{th:subsequence}(b), we deduce that
\[
l_*=\lim_{j\to\infty}\F(x^{t_j},{\bm\xi}^{t_j},u^{t_j}) \le \F(x^*,{\bm\xi}^*,u^*).
\]
Since the converse inequality is an immediate consequence of the lower semicontinuity of $\F$, we conclude that $\F \equiv l_*$ on $\Omega$.

Now, the global convergence of $\{x^t\}$ can be proved based on $\F \equiv l_*$ on the nonempty compact set $\Omega\subseteq {\rm dom}\partial \F$, \eqref{decreas:F}, \eqref{parti:bound}, Lemma~\ref{uni-kl} and the KL assumption on $\F$. The proof is routine (see, for example, \cite[Theorem~1]{BolteSabach14}) and we omit the proof for brevity.
\end{proof}

One can also establish local convergence rate of the sequence generated by \GP{} with $M = 0$ based on the assumptions of Theorem~\ref{KL_nonconvex} and the additional assumption that $\F$ is a KL function with exponent $\theta\in [0,1)$; we refer the readers to \cite[Theorem~3.4]{AtBoReSo10} for a similar analysis.

\subsection{Relating the KL exponent of $\Psi$ and $\F$ when $C = \R^n$}\label{sec5.3}

Notice that global convergence of the sequence generated by \GP{} with $M = 0$ was established in Theorems~\ref{thm5.1} and \ref{KL_nonconvex} under two different KL assumptions. Theorem~\ref{thm5.1} studied the case when each $D_i$ is convex and requires $\Psi+\delta_C$ in \eqref{sum_subproblem_concave} to be a KL function, while Theorem~\ref{KL_nonconvex} studied the case when some $D_i$'s are possibly nonconvex and $C = \R^n$, and requires $\F$ in \eqref{def:F} to be a KL function. In this section, we study a relationship between these two KL assumptions.

We start with the following lemma, which describes how $\F$ is related to the stationary points of $\Psi$ in \eqref{sum_subproblem_concave} when $C = \R^n$ and each $D_i$ is convex.

\begin{lemma}\label{lem:5.4}
Suppose that each $D_i$ is convex, $C = \R^n$, $\Psi$ is defined in \eqref{sum_subproblem_concave} with $\psi\in \Theta$ and $\F$ is defined in \eqref{def:F}. If $0\in \partial \F(\bar{x}, \bar{{\bm\xi}}, \bar{u})$, then $\F(\bar{x}, \bar{{\bm\xi}}, \bar{u}) =  \Psi(\bar{x})$, $\nabla \Psi(\bar x) = 0$ and
\begin{align}\label{eq:stationary}
\sum_{i=1}^{m} \bar{u}_i(\bar{x}- \bar{{\bm\xi}}_i) = 0,\  \bar{{\bm\xi}}=[P_{D_i}(\bar{x})]_{i=1}^m\ {\rm and}\ \bar{u}=[\rho'(d^2(\bar{x},D_i))]_{i=1}^m < 0.
\end{align}
\end{lemma}
\begin{proof}
We first prove \eqref{eq:stationary}. Since $0\in \partial \F(\bar{x}, \bar{{\bm\xi}}, \bar{u})$, we see from \eqref{partial:F} that $\bar{{\bm\xi}}_i \in D_i$, $0=\sum_{i=1}^{m} \bar{u}_i(\bar{x}- \bar{{\bm\xi}}_i)$, $\|\bar{x}- \bar{{\bm\xi}}_i\|^2\in \partial\rho^*(\bar{u}_i)$ and $- 2\bar{u}_i (\bar{x} - \bar{{\bm\xi}}_i) \in  N_{D_i}(\bar{{\bm\xi}}_i)$ for each $i$. Hence, the first relation in \eqref{eq:stationary} holds. Also, combining $\|\bar{x}- \bar{{\bm\xi}}_i\|^2\in \partial\rho^*(\bar{u}_i)$ with \cite[Proposition~11.3]{Rock_wets98}, we obtain
\begin{equation}\label{baru}
\bar{u}_i = \rho'(\|\bar{x} - \bar{{\bm\xi}}_i\|^2)< 0.
\end{equation}
Using this and $- 2\bar{u}_i (\bar{x} - \bar{{\bm\xi}}_i) \in  N_{D_i}(\bar{{\bm\xi}}_i)$, we have $\bar{x} - \bar{{\bm\xi}}_i \in  N_{D_i}(\bar{{\bm\xi}}_i)$. This together with \cite[Proposition~6.17]{Rock_wets98} and the convexity of $D_i$ implies that
\begin{align*}
\bar{{\bm\xi}}_i = (I+N_{D_i})^{-1}(\bar{x})= P_{D_i}(\bar{x}).
\end{align*}
In particular, the second relation in \eqref{eq:stationary} holds and $\|\bar{x} - \bar{{\bm\xi}}_i\|^2 = d^2(\bar{x},D_i)$. Combining this with \eqref{baru} gives the third relation in \eqref{eq:stationary}. These prove \eqref{eq:stationary}.

Next, we deduce from \eqref{eq:psiand F}, \eqref{eq:stationary} and the definitions of $\widehat{\bm\xi}$ and $\widehat u$ that $\F(\bar{x}, \bar{{\bm\xi}}, \bar{u}) =  \Psi(\bar{x})$.
Finally, from \eqref{eq:stationary} and \eqref{rho'} we see that
\[
0 = \sum_{i=1}^{m} \bar{u}_i(\bar{x}- \bar{{\bm\xi}}_i) = \sum_{i=1}^{m} \rho'(d^2(\bar{x},D_i))(\bar{x}- P_{D_i}(\bar{x}))= \sum_{i=1}^{m} -\psi_+'(d^2(\bar{x},D_i))(\bar{x}- P_{D_i}(\bar{x})),
\]
i.e., $\nabla \Psi(\bar x) = 0$. This completes the proof.
\end{proof}

We now present our analysis concerning the two different KL assumptions used in Theorems~\ref{thm5.1} and \ref{KL_nonconvex}. In our analysis below, we assume that each $D_i$ is convex, $C = \R^n$ and $\psi\in \Theta$. We also require in addition that $\psi$ is \textbf{strict concave} on $\R_{+}$. This latter assumption together with \eqref{liftingrho} shows that $\rho$ is a \textbf{strictly convex} continuously differentiable nonincreasing function on $\R$. One can check that the functions $s\mapsto  \frac{s}{s+\epsilon} + \epsilon s$, $\epsilon > 0$, and $s\mapsto \log(s+\epsilon) - \log (\epsilon)$, $\epsilon\in (0,1)$, discussed in the beginning of Section~\ref{sec4} are both strictly concave on $\R_+$.

Under the additional strict concavity assumption on $\psi$, we see from \cite[Theorem~26.3]{Roc70} that $\rho^*$ is essentially smooth. Thus, ${\rm dom}\,\partial\rho^*$ is open thanks to \cite[Theorem~26.1]{Roc70} and $\rho^*$ is indeed continuously differentiable at each $u_i\in {\rm dom}\,\partial \rho^*$ in view of \cite[Theorem~25.5]{Roc70}. On the other hand, notice from \eqref{partial:F} that $(x,{\bm\xi},u)\in {\rm dom}\,\partial \F$ implies $u_i \in {\rm dom}\,\partial \rho^*$ for each $i$. Thus, if $(x,{\bm\xi},u)\in {\rm dom}\,\partial \F$, then $\rho^*$ is continuously differentiable at $u_i$ for each $i$. In the next lemma, we establish some inequalities concerning $(\rho^*)'$ at some special points in $\rm{dom}\,\partial \F$.

\begin{lemma}\label{lem:5.5}
Suppose that each $D_i$ is convex, $C = \R^n$, $\psi\in \Theta$ and is strictly concave on $\R_{+}$, $\Psi$ is defined in \eqref{sum_subproblem_concave} and $\F$ is defined in \eqref{def:F}. Let $0\in \partial \F(\bar{x}, \bar{{\bm\xi}}, \bar{u})$. Then there exist positive numbers $\epsilon$, $L$, $\bar c$, $c_1$ and $c_2$ so that whenever $(x,{\bm\xi},u)\in {\rm dom}\,\partial \F\cap {\bf B}((\bar x,\bar{\bm\xi},\bar u),\epsilon)$, the following inequalities hold:
\begin{equation}\label{ineq:fandpsi}
 \F(x,{\bm\xi},u) \leq \Psi(x)+ \frac{L}{2}\sum_{i=1}^{m}[d^2_{D_i}(x) - (\rho^*)^{'}(u_i)]^2 - \sum_{i=1}^{m}u_i\left[\norm{x- {\bm\xi}_i}^2 - d^2_{D_i}(x)\right],
\end{equation}
\begin{equation}\label{con(c)}
 L|(\rho^*)^{'}(u_i) - d^2_{D_i}(x) |\ge | u_i + \psi'_+(d^2_{D_i}(x))|\text{ for all $i$},
\end{equation}
\begin{equation}\label{ineq:dis}
 |d^2_{D_i}(x)- (\rho^*)^{'}(u_i)|< m^{-1}\text{ and } 0\leq \|x-{\bm\xi}_i\|^2 - d^2_{D_i}(x) < m^{-1}\text{ for all $i$},
\end{equation}
\begin{equation}\label{ineq:u_i}
 -\bar c < u_i < 0 \text{ and }u_i^2\ge c_1 \text{ for all $i$},
\end{equation}
\begin{equation}\label{con(a)}
 \inf_{\bm\mu_i\in N_{D_i}({\bm\xi}_i)}\,\| -2u_i({\bm\xi}_i-x) + \bm\mu_i\|^2 \geq c_1\left(\|{\bm\xi}_i - x\|^2 - d^2_{D_i}(x) \right)\geq 0\text{ for all $i$},
\end{equation}
\begin{equation}\label{con(b)}
 \inf_{\bm\mu_i\in N_{D_i}({\bm\xi}_i)}\,\| -2u_i({\bm\xi}_i-x) + \bm\mu_i\|^2 \geq c_2 \|P_{D_i}(x) -{\bm\xi}_i\|^2\text{ for all $i$}.
\end{equation}
\end{lemma}
\begin{proof}
Since $(\bar x,\bar{\bm\xi},\bar u)\in {\rm dom}\,\partial \F$, from the discussion preceding this lemma, we see that there exists $\epsilon_0>0$ so that $(\rho^*)^{'}$ is continuous at $u_i$ for all $i$ whenever $(x,{\bm\xi},u)\in {\rm dom}\,\partial \F\cap {\bf B}((\bar x,\bar{\bm\xi},\bar u),\epsilon_0)$.

Let $(x,{\bm\xi},u)\in {\rm dom}\,\partial \F\cap {\bf B}((\bar x,\bar{\bm\xi},\bar u),\epsilon_0)$ and define $\tilde{y}_i = (\rho^*)^{'}(u_i)$ for each $i$. Then, from \cite[Proposition~11.3]{Rock_wets98}, we see that
\begin{align}\label{uiandyi}
u_i = \rho'(\tilde{y}_i)\text{~ and ~} \rho^* (u_i) = -\rho(\tilde{y}_i) + \tilde{y}_i u_i \text{~ for each ~}i.
\end{align}
Moreover, since $\rho'$ is Lipschitz continuous on $\R$ in view of $\psi\in \Theta$ and \eqref{rho'}, we have
\begin{equation}\label{hahaha}
  \rho(d^2_{D_i}(x)) \leq \rho(\tilde{y}_i) + u_i (d^2_{D_i}(x) - \tilde{y}_i) + \frac{L}{2}(d^2_{D_i}(x) - \tilde{y}_i)^2,
\end{equation}
where $L$ is the Lipschitz continuity modulus.
Then we have
\begin{align*}
\begin{aligned}
&\F(x,{\bm\xi},u) = \sum_{i=1}^{m}\left[ - u_i\norm{x- {\bm\xi}_i}^2 - \rho(\tilde{y}_i) +\tilde{y}_i u_i \right] \\
&= \sum_{i=1}^{m}\left[ - \rho(\tilde{y}_i) - u_i(d^2_{D_i}(x) - \tilde{y}_i) - u_i(\norm{x- {\bm\xi}_i}^2 - d^2_{D_i}(x)) \right]\\
&\leq \sum_{i=1}^{m}\left[ -\rho(d^2_{D_i}(x)) + \frac{L}{2}(d^2_{D_i}(x) - \tilde{y}_i)^2 \right] + \sum_{i=1}^{m} (- u_i)\left[\norm{x- {\bm\xi}_i}^2 - d^2_{D_i}(x)\right]\\
&= \Psi(x)+ \frac{L}{2}\sum_{i=1}^{m}[d^2_{D_i}(x) - (\rho^*)^{'}(u_i)]^2 + \sum_{i=1}^{m}(-u_i)\left[\norm{x- {\bm\xi}_i}^2 - d^2_{D_i}(x)\right],
\end{aligned}
\end{align*}
where the first equality follows from the second relation in \eqref{uiandyi}, the inequality is due to \eqref{hahaha} and the last equality follows from the definitions of $\Psi$ and $\tilde y_i$.
Furthermore, we deduce that for each $i$,
\begin{align*}
|u_i + \psi'_+(d^2_{D_i}(x))| = |\rho'(\tilde{y}_i) - \rho'(d^2_{D_i}(x))| \leq L |\tilde{y}_i -  d^2_{D_i}(x)| = L | (\rho^*)^{'}(u_i) -  d^2_{D_i}(x)|,
\end{align*}
where the first equality follows from the first relation in \eqref{uiandyi} and the expression of $\rho'$ in \eqref{rho'}, the inequality follows from the Lipschitz continuity of $\rho'$ and the last equality follows from the definition of $\tilde y_i$.
Thus, \eqref{ineq:fandpsi} and \eqref{con(c)} hold whenever $(x,{\bm\xi},u)\in {\rm dom}\,\partial \F\cap {\bf B}((\bar x,\bar{\bm\xi},\bar u),\epsilon_0)$.

Next, note that \eqref{eq:stationary} holds because $0\in \partial \F(\bar{x}, \bar{{\bm\xi}}, \bar{u})$. Then for each $i$, it holds that
\begin{equation}\label{limit0}
\lim_{u_i\to \bar u_i}(\rho^*)'(u_i) = (\rho^*)'(\bar u_i) = (\rho^*)^{'}\left(\rho'(d^2_{D_i}(\bar{x}))\right)= d^2_{D_i}(\bar{x}) = \lim_{x\to \bar x}d^2_{D_i}(x),
\end{equation}
where the second equality follows from the last relation in \eqref{eq:stationary} and the third equality follows from \cite[Proposition~11.3]{Rock_wets98}. Moreover, we have for each $i$,
\begin{equation}\label{limit1}
\lim_{(x,{\bm\xi}_i)\to (\bar x,\bar{\bm\xi}_i)}\|x-{\bm\xi}_i\|^2 = \|\bar x - \bar{\bm\xi}_i\|^2 = d^2_{D_i}(\bar x) = \lim_{x\to \bar x}d^2_{D_i}(x),
\end{equation}
where the second equality above follows from the second relation in \eqref{eq:stationary}. In addition, for any $(x,{\bm\xi},u)\in {\rm dom}\,\partial \F$, we have ${\bm\xi}_i\in D_i$ for all $i$ and hence $\|x-{\bm\xi}_i\|^2\ge d^2_{D_i}(x)$. In view of this, \eqref{limit0} and \eqref{limit1}, we conclude that one can further shrink $\epsilon_0$ so that \eqref{ineq:dis} also holds whenever $(x,{\bm\xi},u)\in {\rm dom}\,\partial \F\cap {\bf B}((\bar x,\bar{\bm\xi},\bar u),\epsilon_0)$. Now, recall that $\bar u < 0$ from the third relation in \eqref{eq:stationary}. Thus, one can shrink $\epsilon_0$ further so that \eqref{ineq:u_i} holds true for some positive numbers $c_1$ and $\bar c$ whenever $(x,{\bm\xi},u)\in {\rm dom}\,\partial \F\cap {\bf B}((\bar x,\bar{\bm\xi},\bar u),\epsilon_0)$.

It now remains to prove \eqref{con(a)} and \eqref{con(b)}. We first prove \eqref{con(a)}. Let $(x,{\bm\xi},u)\in {\rm dom}\,\partial \F\cap {\bf B}((\bar x,\bar{\bm\xi},\bar u),\epsilon_0)$. Fix any $i$. Then we have $u_i\le -\sqrt{c_1}$ from \eqref{ineq:u_i} and hence
\begin{align}\label{con1}
\begin{split}
&\inf _{\bm\mu_i\in N_{D_i}({\bm\xi}_i)}\,\| -2u_i({\bm\xi}_i-x) + \bm\mu_i\|^2 = 4u_i^2\cdot\inf_{-\frac{1}{2u_i}\bm\mu_i\in N_{D_i}({\bm\xi}_i)}\,\left\| {\bm\xi}_i-x -\frac{1}{2u_i}\bm\mu_i\right\|^2 \\
&\ge u_i^2\cdot\inf_{\bm\mu_i\in N_{D_i}({\bm\xi}_i)}\,\| {\bm\xi}_i-x +\bm\mu_i\|^2 \ge c_1 \dist^2 \left(0, \partial f ({\bm\xi}_i)\right),
\end{split}
\end{align}
where $f(\cdot):=\frac{1}{2}\|\cdot - x\|^2 + \delta_{D_i}(\cdot)$. Since $D_i$ is convex, we see that $f$ is a strongly convex function with modulus $1$. This together with $P_{D_i}(x) \in D_i$ and ${\bm\xi}_i \in D_i$ gives
\[
\frac{1}{2}\|P_{D_i}(x) - x\|^2 - \frac{1}{2}\|{\bm\xi}_i - x\|^2 \geq \inner{\eta}{P_{D_i}(x) - {\bm\xi}_i} + \frac{1}{2}\|P_{D_i}(x) - {\bm\xi}_i\|^2\text{~ for all ~}\eta \in \partial f({\bm\xi}_i).
\]
Multiplying the above inequality by 2 and rearranging terms, we obtain further that
\begin{align*}
2\inner{\eta}{{\bm\xi}_i - P_{D_i}(x)}\geq \|{\bm\xi}_i - x\|^2 -\|P_{D_i}(x) - x\|^2 + \|P_{D_i}(x) - {\bm\xi}_i\|^2.
\end{align*}
This together with the relation $\|\eta\|^2 + \|{\bm\xi}_i - P_{D_i}(x) \|^2 \geq 2\|\eta\| \|{\bm\xi}_i - P_{D_i}(x)\|$ gives
\begin{align*}
\|\eta\|^2 \geq \|{\bm\xi}_i - x\|^2 - \|P_{D_i}(x) - x\|^2\geq 0 \text{~ for all ~}\eta \in \partial f({\bm\xi}_i),
\end{align*}
where the last inequality holds because ${\bm\xi}_i \in D_i$.
This together with \eqref{con1} shows that \eqref{con(a)} holds whenever $(x,{\bm\xi},u)\in {\rm dom}\,\partial \F\cap {\bf B}((\bar x,\bar{\bm\xi},\bar u),\epsilon_0)$.

Finally, we show that there exists $\epsilon \in (0,\epsilon_0]$ so that \eqref{con(b)} holds whenever $(x,{\bm\xi},u)\in {\rm dom}\,\partial \F\cap {\bf B}((\bar x,\bar{\bm\xi},\bar u),\epsilon)$. To this end, we first recall from the second relation in \eqref{eq:stationary} that $\bar {\bm\xi}_i = P_{D_i}(\bar x)$ for each $i$. Moreover, recall also from \cite[Proposition~6.17]{Rock_wets98} that $(I + N_{D_i})^{-1} = P_{D_i}$, since $D_i$ is convex. Using these together with the nonexpansiveness (and hence, Lipschitz continuity) of $P_{D_i}$, we see that $(I + N_{D_i})^{-1}$ has the Aubin property at $(\bar x,\bar{\bm\xi}_i)$. Hence, $I + N_{D_i}$ is metrically regular at $(\bar{\bm\xi}_i, \bar x)$ thanks to \cite[Theorem 9.43]{Rock_wets98}. Thus, there exist $\epsilon_i >0$ and $\kappa_i>0$ such that
\begin{equation}\label{ineq:aubin}
\kappa_i\|{\bm\xi}_i - (I + N_{D_i})^{-1}(x)\| \leq \inf_{\bm\omega_i\in (I + N_{D_i})({\bm\xi}_i)}\|x - \bm\omega_i\|\text{~ for all ~}(x,{\bm\xi}_i)\in {\bf B}((\bar x,\bar{\bm\xi}_i),\epsilon_i).
\end{equation}
On the other hand, let $(x,{\bm\xi},u)\in {\rm dom}\,\partial \F\cap {\bf B}((\bar x,\bar{\bm\xi},\bar u),\epsilon_0)$. Then, from \eqref{con1}, we see that $\inf_{\bm\mu_i\in N_{D_i}({\bm\xi}_i)}\,\| -2u_i({\bm\xi}_i-x) + \bm\mu_i\|^2$ is bounded below by
\begin{align*}
 c_1 \inf_{\bm\mu_i\in N_{D_i}({\bm\xi}_i)}\,\| x - {\bm\xi}_i - \bm\mu_i\|^2 = c_1 \inf_{\bm\omega_i\in (I + N_{D_i})({\bm\xi}_i)}\,\| x - \bm\omega_i\|^2.
\end{align*}
Combining this with \eqref{ineq:aubin} and recalling that $(I + N_{D_i})^{-1} = P_{D_i}$, we conclude that \eqref{con(b)} holds for some $c_2 > 0$ whenever $(x,{\bm\xi},u)\in {\rm dom}\,\partial \F\cap {\bf B}((\bar x,\bar{\bm\xi},\bar u),\epsilon)$, where $\epsilon := \min\{\epsilon_i:0\le i\le m\}$. This completes the proof.
\end{proof}

We are now ready to prove the main theorem of this section.

\begin{theorem}
Suppose that each $D_i$ is convex, $C = \R^n$, $\psi\in \Theta$ and is strictly concave on $\R_{+}$. If $\Psi$ in \eqref{sum_subproblem_concave} satisfies the KL property with exponent $\theta \in [\frac12,1)$, then $\F$ in \eqref{def:F} satisfies the KL property with exponent $\theta$.
\end{theorem}
\begin{proof}
In view of \cite[Lemma~2.1]{LiPong17}, it suffices to prove that $\F$ satisfies the KL property with exponent $\theta$ at all points $(\bar x,\bar{{\bm\xi}}, \bar{u})$ verifying $0 \in \partial \F(\bar x,\bar{{\bm\xi}}, \bar{u})$. To this end, fix any $(\bar{x},\bar{{\bm\xi}}, \bar{u})$ that satisfies $0 \in \partial \F(\bar x,\bar{{\bm\xi}}, \bar{u})$. Notice that $\Psi$ is continuously differentiable because each $D_i$ is convex. Since $\Psi$ is also a KL function with exponent $\theta$, we see that there exist positive numbers $\eta$ and $\epsilon$ such that
\begin{align}\label{KLpsi}
\eta\|\nabla \Psi(x)\|^{\frac{1}{\theta}}\geq \Psi(x)- \Psi(\bar{x})\text{~ whenever ~} \|x - \bar{x}\| \leq \epsilon;
\end{align}
here, the condition on the bound on function values is waived by the continuity of $\Psi$ and by choosing a smaller $\epsilon$ if necessary.
In view of Lemma~\ref{lem:5.5}, we can shrink this $\epsilon$ further so that \eqref{ineq:fandpsi}, \eqref{con(c)}, \eqref{ineq:dis}, \eqref{ineq:u_i}, \eqref{con(a)} and \eqref{con(b)} hold whenever $(x,{\bm\xi},u)\in {\rm dom}\,\partial \F\cap {\bf B}((\bar x,\bar{\bm\xi},\bar u),\epsilon)$.

Now, fix any $(x,{\bm\xi},u)\in {\rm dom}\,\partial \F\cap {\bf B}((\bar x,\bar{\bm\xi},\bar u),\epsilon)$. By \cite[Lemma~2.2]{LiPong17} and a suitable scaling, there exists $c_3 > 0$ so that $\dist^{\frac{1}{\theta}}(0, \partial \F(x,{\bm\xi}, u))$ is bounded below by
\begin{align}\label{ineq:dist1}
c_3 \left[\dist^{\frac{1}{\theta}}(0,\partial_x \F(x,{\bm\xi}, u)) + 3\dist^{\frac{1}{\theta}}(0,\partial_{{\bm\xi}} \F(x,{\bm\xi}, u)) + \dist^{\frac{1}{\theta}}(0,\partial_u \F(x,{\bm\xi}, u))\right].
\end{align}
We now derive lower bounds for the three terms on the right hand side of \eqref{ineq:dist1}. We first derive bounds for the second term.
Using \cite[Lemma~2.2]{LiPong17} and \eqref{partial:F}, we see that there exists $c_0 >0$ so that 
\begin{align}\label{ineq:con}
\begin{split}
&2\dist^{\frac{1}{\theta}}(0,\partial_{{\bm\xi}} \F(x,{\bm\xi}, u)) = 2\left(\sum_{i=1}^m \inf_{\bm\mu_i \in N_{D_i}({\bm\xi}_i)}\|- 2u_i ({\bm\xi}_i - x) + \bm\mu_i \|^2\right)^{\frac{1}{2\theta}}\\
&\geq 2c_0 \sum_{i=1}^m \inf_{\bm\mu_i \in N_{D_i}({\bm\xi}_i)}\|- 2u_i ({\bm\xi}_i - x) + \bm\mu_i \|^{\frac{1}{\theta}}\\
&\geq c_0 \sum_{i=1}^m c_1^{\frac{1}{2\theta}}\left(\|{\bm\xi}_i - x\|^2 - d^2_{D_i}(x) \right)^{\frac{1}{2\theta}} + c_0 \sum_{i=1}^m c_2^{\frac{1}{2\theta}}\|P_{D_i}(x) -{\bm\xi}_i\|^{\frac{1}{\theta}}\\
&\geq \bar \kappa_0\left(\sum_{i=1}^m \left(\|{\bm\xi}_i - x\|^2 - d^2_{D_i}(x) \right)^{\frac{1}{\theta}} + \sum_{i=1}^m\|P_{D_i}(x) -{\bm\xi}_i\|^{\frac{1}{\theta}}\right),
\end{split}
\end{align}
where $\bar \kappa_0 := \min\{c_0 c_1^{\frac{1}{2\theta}},c_0 c_2^{\frac{1}{2\theta}}\}$, the second inequality follows from \eqref{con(a)} and \eqref{con(b)}, and the last inequality holds because $0\leq \|{\bm\xi}_i - x\|^2 - d^2_{D_i}(x)<1$ (see \eqref{ineq:dis}).
On the other hand, using \eqref{con(a)}, we also have 
\begin{align}\label{ineq:cone1}
\begin{split}
&\dist^{\frac{1}{\theta}}(0,\partial_{{\bm\xi}} \F(x,{\bm\xi}, u)) = \left(\sum_{i=1}^m \inf_{\mu_i \in N_{D_i}({\bm\xi}_i)}\|- 2u_i ({\bm\xi}_i - x) + \mu_i \|^2\right)^{\frac{1}{2\theta}} \\
&\geq \left[\sum_{i=1}^m c_1\left(\|{\bm\xi}_i - x\|^2 - d^2_{D_i}(x) \right)\right]^{\frac{1}{2\theta}}\geq c_1^{\frac{1}{2\theta}}\sum_{i=1}^m \left(\|{\bm\xi}_i - x\|^2 - d^2_{D_i}(x) \right),
\end{split}
\end{align}
where the last inequality follows from $0\leq \sum_{i=1}^m \left(\|{\bm\xi}_i - x\|^2 - d^2_{D_i}(x) \right)<1$ (see \eqref{ineq:dis}) and $\theta\in [\frac12,1)$.

We next derive a lower bound for the third term on the right hand side in \eqref{ineq:dist1}. In view of \cite[Lemma~2.2]{LiPong17}, there exists $\bar\kappa_1 \in (0,\bar\kappa_0)$, with $\bar \kappa_0$ given in \eqref{ineq:con}, so that
\begin{align}\label{ineq:rhoumain}
\begin{split}
&\dist^{\frac{1}{\theta}}(0,\partial_u \F(x,{\bm\xi}, u)) \ge \bar\kappa_1\sum_{i=1}^m \left|(\rho^*)^{'}(u_i) - \|x- {\bm\xi}_i\|^2 \right|^{\frac{1}{\theta}}\\
&\geq \bar\tau_1\sum_{i=1}^m \left|(\rho^*)^{'}(u_i) - d^2_{D_i}(x)\right|^{\frac{1}{\theta}} - \bar \kappa_0\eta_1 \sum_{i=1}^m\left( \|x- {\bm\xi}_i\|^2 - d^2_{D_i}(x) \right)^{\frac{1}{\theta}},
\end{split}
\end{align}
where the last inequality follows from \cite[Lemma~3.1]{LiPong17} for some $\eta_1 \in (0,1)$ and $\bar \tau_1 > 0$.

Finally, we derive a lower bound for the first term on the right hand side in \eqref{ineq:dist1}. To this end, observe that
\begin{equation}\label{underlinec}
\widehat{c} := \sup\left\{(|a_1| + \cdots + |a_m|)^{\frac{1}{\theta}}:\; |a_1|^\frac1\theta + \cdots + |a_m|^\frac1\theta=1\right\} \in (0,\infty).
\end{equation}
Choose $\bar\kappa_2 = \min\{2^\frac1\theta,{\bar c}^{-\frac1\theta}\widehat{c}^{-1}\bar\kappa_0\}$ (with $\bar c$ and $\bar \kappa_0$ given in \eqref{ineq:u_i} and \eqref{ineq:con}, respectively), we then have
\begin{align}\label{ineq:uandxi}
\begin{split}
& \dist^{\frac{1}{\theta}}(0,\partial_x \F(x,{\bm\xi}, u)) = \left\|-2\sum_{i=1}^{m} u_i(x- {\bm\xi}_i)\right\|^{\frac{1}{\theta}}\ge \bar\kappa_2\left\|\sum_{i=1}^{m} u_i(x- {\bm\xi}_i)\right\|^{\frac{1}{\theta}}\\
&\geq  \bar\tau_2 \left\|\sum_{i=1}^{m} u_i(x- P_{D_i}(x)) \right\|^{\frac{1}{\theta}} - \bar\kappa_2\eta_2\left[\sum_{i=1}^{m} |u_i|\|P_{D_i}(x) -{\bm\xi}_i\|\right]^{\frac{1}{\theta}}\\
&\geq  \bar\tau_2 \left\|\sum_{i=1}^{m} u_i(x- P_{D_i}(x)) \right\|^{\frac{1}{\theta}} - \bar\kappa_0\eta_2\sum_{i=1}^{m}\|P_{D_i}(x) -{\bm\xi}_i\|^{\frac{1}{\theta}},
\end{split}
\end{align}
where the second inequality follows from \cite[Lemma~3.1]{LiPong17} for some $\eta_2 \in (0,1)$ and $\bar\tau_2 > 0$ and the triangle inequality, and the last inequality follows from \eqref{ineq:u_i}, the definition of $\widehat{c}$ and the fact that $\bar\kappa_2 \le {\bar c}^{-\frac1\theta}\widehat{c}^{-1}\bar\kappa_0$.

Now, we derive a lower bound for the first term on the right hand side of \eqref{ineq:uandxi}. Let $\widehat M := \max\limits_{1\le i\le m}\max_{x\in {\bf B}(\bar x,\epsilon)}\|x - P_{D_i}(x)\|$. Choose $\bar \kappa_3 = \min\{2^{-\frac1\theta}\bar\tau_2,{(2\widehat M L)}^{-\frac1\theta}\widehat{c}^{-1}\bar\tau_1\}$ (with $L$, $\bar \tau_1$ and $\widehat{c}$ given in \eqref{con(c)}, \eqref{ineq:rhoumain} and \eqref{underlinec}, respectively), then we have
\begin{align}\label{mingzi}
\begin{aligned}
 &\bar\tau_2\left\|\sum_{i=1}^{m} u_i(x\!-\! P_{D_i}(x)) \right\|^{\frac{1}{\theta}} \!\!\!\!= \frac{\bar\tau_2}{2^\frac1\theta}\!\left\|\nabla \Psi(x) - 2\sum_{i=1}^{m} (u_i + \psi_+'(d^2_{D_i}(x)))(x\!-\! P_{D_i}(x)) \right\|^{\frac{1}{\theta}}\\
 &\ge \bar \kappa_3 \left\|\nabla \Psi(x) - 2\sum_{i=1}^{m} (u_i + \psi_+'(d^2_{D_i}(x)))(x- P_{D_i}(x)) \right\|^{\frac{1}{\theta}}\\
&\geq \bar \tau_3\|\nabla \Psi(x)\|^{\frac{1}{\theta}} - 2^\frac1\theta\bar \kappa_3\eta_3 \left(\sum_{i=1}^{m}|u_i + \psi_+'(d^2_{D_i}(x))|\cdot\|x- P_{D_i}(x)\| \right)^{\frac{1}{\theta}}\\
&\geq \bar \tau_3\|\nabla \Psi(x)\|^{\frac{1}{\theta}} - \bar \kappa_3\eta_3{(2\widehat M)}^\frac1\theta\widehat{c} \sum_{i=1}^{m}|u_i + \psi_+'(d^2_{D_i}(x))|^{\frac{1}{\theta}}\\
&\geq \bar \tau_3\|\nabla \Psi(x)\|^{\frac{1}{\theta}} - \bar \tau_1\eta_3\sum_{i=1}^{m}\left|(\rho^*)^{'}(u_i) - d^2_{D_i}(x)\right|^{\frac{1}{\theta}},
\end{aligned}
\end{align}
where the second inequality follows from \cite[Lemma~3.1]{LiPong17} for some $\eta_3 \in (0,1)$ and $\bar \tau_3 > 0$ and the triangle inequality, the third inequality follows from the definitions of $\widehat{c}$ and $\widehat M$, and the last inequality follows from \eqref{con(c)} and the fact that $\bar \kappa_3 \le {(2\widehat M L)}^{-\frac1\theta}\widehat{c}^{-1}\bar\tau_1$.

Combining \eqref{ineq:con}, \eqref{ineq:cone1}, \eqref{ineq:rhoumain}, \eqref{ineq:uandxi} and \eqref{mingzi} with the lower bound of the quantity $\dist^{\frac{1}{\theta}}(0, \partial \F(x,{\bm\xi}, u))$ in \eqref{ineq:dist1}, we see that 
\begin{align*}
  \begin{aligned}
  & c_3^{-1}\dist^{\frac{1}{\theta}}(0, \partial \F(x,{\bm\xi}, u))\\
  &\ge \bar \kappa_0\left(\sum_{i=1}^m\! \left(\|{\bm\xi}_i - x\|^2\! -\! d^2_{D_i}(x) \right)^{\frac{1}{\theta}}\!+ \!\!\sum_{i=1}^m\!\|P_{D_i}(x) -{\bm\xi}_i\|^{\frac{1}{\theta}}\right)\!+ \!c_1^{\frac{1}{2\theta}}\!\sum_{i=1}^m \left(\|{\bm\xi}_i - x\|^2\! - \!d^2_{D_i}(x) \right)\\
    & \ \ \ + \bar\tau_1\sum_{i=1}^m \left|(\rho^*)^{'}(u_i) - d^2_{D_i}(x)\right|^{\frac{1}{\theta}} - \bar \kappa_0\eta_1 \sum_{i=1}^m\left( \|x- {\bm\xi}_i\|^2 - d^2_{D_i}(x) \right)^{\frac{1}{\theta}}\\
    & \ \ \ + \bar \tau_3\|\nabla \Psi(x)\|^{\frac{1}{\theta}} - \bar \tau_1\eta_3\sum_{i=1}^{m}\left|(\rho^*)^{'}(u_i) - d^2_{D_i}(x)\right|^{\frac{1}{\theta}}- \bar\kappa_0\eta_2\sum_{i=1}^{m}\|P_{D_i}(x) -{\bm\xi}_i\|^{\frac{1}{\theta}}
  \end{aligned}
\end{align*}
Grouping like terms and noting that $\eta_1,\eta_2,\eta_3\in (0,1)$ and $\|{\bm\xi}_i - x\|^2 - d^2_{D_i}(x)\ge 0$ (thanks to \eqref{con(a)}), we see further that
\begin{align*}
  \begin{aligned}
  &c_3^{-1}\dist^{\frac{1}{\theta}}(0, \partial \F(x,{\bm\xi}, u))\\
   &\ge c_4\left[\eta \|\nabla \Psi(x)\|^{\frac{1}{\theta}} + \frac{L}2\sum_{i=1}^m \left|(\rho^*)^{'}(u_i) - d^2_{D_i}(x)\right|^{\frac{1}{\theta}} + \bar{c}\sum_{i=1}^m\left(\|{\bm\xi}_i - x\|^2 - d^2_{D_i}(x) \right)\right]\\
    & \ge c_4\left[\eta \|\nabla \Psi(x)\|^{\frac{1}{\theta}} + \frac{L}2\sum_{i=1}^m \left|(\rho^*)^{'}(u_i) - d^2_{D_i}(x)\right|^{\frac{1}{\theta}} + \sum_{i=1}^m(-u_i)\left(\|{\bm\xi}_i - x\|^2 - d^2_{D_i}(x) \right)\right]\\
    & \ge c_4\left[\Psi(x)- \Psi(\bar{x}) + \frac{L}{2} \sum_{i=1}^m \left|(\rho^*)^{'}(u_i) - d^2_{D_i}(x)\right|^{\frac{1}{\theta}}\!\!\! + \sum_{i=1}^m (-u_i)\left(\|{\bm\xi}_i - x\|^2 - d^2_{D_i}(x) \right)\right]\\
    & \ge c_4 \left(\F(x,{\bm\xi},u) - \F(\bar{x}, \bar{{\bm\xi}}, \bar{u})\right),
  \end{aligned}
\end{align*}
where: the first inequality holds for some $c_4 > 0$ upon a suitable scaling (recall that $\eta$ is given in \eqref{KLpsi}, $\bar c$ is defined in \eqref{ineq:u_i} and $L$ is given in \eqref{ineq:fandpsi}); the second inequality follows from \eqref{ineq:u_i}, the third inequality follows from \eqref{KLpsi}, and the last inequality follows from \eqref{ineq:fandpsi} and $\Psi(\bar{x}) = \F(\bar{x}, \bar{{\bm\xi}}, \bar{u})$, thanks to Lemma~\ref{lem:5.4} and the assumption that $0\in \partial \F(\bar x,\bar{\bm\xi},\bar u)$. This completes the proof.
\end{proof}

\section{Numerical test}\label{sec6}

In this section, we perform numerical experiments to study the performance of \EAS{} on some large-scale MFS$_C$ problems. All codes are written in Matlab, and the experiments are performed in  Matlab 2015b on a 64-bit PC with an Intel(R) Core(TM) i7-4790 CPU (3.60GHz) and 32GB of RAM.

In \EAS{}, we take $\varphi_{\epsilon_k}(s)=1-\frac{\log(s+\epsilon_k)}{\log\epsilon_k}$ with $\epsilon_k = 0.9(0.1)^{k-1}$, and terminate when $\epsilon_k \le 10^{-6}$. As for initialization, in our experiments below, we use randomly generated initial points $\tilde x_0$, which are projections onto $C$ of random vectors with i.i.d. standard Gaussian entries.

As discussed in Section~\ref{sec4}, the corresponding subproblems of \EAS{} take the form of \eqref{sum_subproblem_concave} with $\psi(s) = \psi_{\epsilon_k}(s) := \log(\epsilon_k + s) - \log (\epsilon_k)$, and we use \GP{} for solving them approximately. In the \GP{}, we pick $\alpha_{\text{min}}= 10^{-10}$, $\alpha_{\text{max}}= 10^{10}$, $\eta =1/2$, $M=9$ and $\sigma=10^{-4}$. We initialize the algorithm at $\tilde x^{k-1}$ for approximately minimizing $\Phi_{\epsilon_k} + \delta_C$, and terminate the algorithm when
$\|x^t-x^{t-1}\|\le \max\left\{\frac{10^{-5}}{3^{k-1}},10^{-7}\right\}\cdot\max\{1,\|x^t\|\}$.
As for $\alpha_t^0$, we initialize it at $\alpha_1^0 = 1$ and set, for $t \ge 1$,
\begin{equation*}
\alpha_t^0=
\begin{cases}
P_{[10^{-10},10^{10}]}\left(\frac{\|x^t - x^{t-1}\|^2}{\inner{x^t - x^{t-1}}{g^t - g^{t-1}}}\right)& {\rm if}\ \inner{x^t - x^{t-1}}{g^t - g^{t-1}} > 10^{-12},\\
P_{[10^{-10},10^{10}]}(2\alpha_{t-1}^0)     & {\rm otherwise};
\end{cases}
\end{equation*}
the choice of $\alpha^0_t$ is motivated by the renowned Barzilai-Borwein stepsize.

We apply \EAS{} as described above to two classes of MFS$_C$ problems. In the first class, we set
\[
C := \{x\in [-r,r]^n:\; \|x\|_0 \le s\}
\]
and $D_i = \{x\in \R^n:\; \langle {\bm a}_i,x\rangle\le b_i\}$, where ${\bm a}_i$ is the $i$th row of an $A\in \R^{m\times n}$. For the second class, we consider the same $C$ as above but we choose
$D_i = \{x\in \R^n:\; \langle {\bm a}_i,x\rangle\le b_i\}\cup \{x\in \R^n:\; \langle {\bm p}_i,x\rangle\le q_i\}$, where ${\bm a}_i$ and ${\bm p}_i$ are the $i$th row of $A\in \R^{m\times n}$ and $P\in \R^{m\times n}$, respectively. The matrices $A$, $P$ and vectors $b$ and $q$ are randomly generated as follows. We first randomly generate $A$ and $P$ to have i.i.d. standard Gaussian entries. We next generate a $\tilde w\in \R^n$ with $s$ i.i.d. standard Gaussian entries at random positions, and project $\tilde w$ onto $[-r,r]^n$ to form $w$. Set $\tilde b = Aw$ and fix a real number $\bar p \in [0,1]$. We then define $b\in \R^n$ by
\[
b_i = \begin{cases}
  \tilde b_i + 0.01 \varepsilon_i & \mbox{ if }i \le \lceil \bar p m\rceil,\\
  \tilde b_i - 50 \varepsilon_i & \mbox{ otherwise},
\end{cases}
\]
where $\varepsilon_i$ are chosen uniformly at random from $[0,1]$. Finally, we set $q = Pw - 50 \iota$ for some random vector $\iota\in [0,1]^m$. By construction, the system $\{C,D_1,\ldots,D_{\lceil \bar p m\rceil}\}$ is feasible and hence the optimal value of \eqref{MFSC} is at least $\lceil \bar p m\rceil$ for these problems. Moreover, the vector $w\in C$ is not in $C\cap \bigcap_{i=1}^m D_i$. Furthermore, the resulting system $\{C,D_1,\ldots,D_m\}$ is conceivably infeasible because of the subtractions of $50\varepsilon$ and $50\iota$.

In our experiments below, for the two classes of problems, we consider $m = 3000$, $5000$, $n = \frac{m}{5}$, $s = \frac{n}{5}$, $\bar p = 0.5$, $0.6$ and $0.7$, and $r = 10^8$. For each class of problems, for each $m$ and $\bar p$, we randomly generate $5$ instances as described above. For each instance, we solve the corresponding MFS$_C$ problem using \EAS{} from $5$ random initial points.\footnote{These are projections onto $C$ of random vectors with i.i.d. standard Gaussian entries. A projection onto $C$ can be computed efficiently according to \cite[Proposition~3.1]{Lu-Zhang 12}.} We report the number of iterations (iter) and the CPU time in seconds (CPU) in Table~\ref{t1}, averaged over the $5$ random initializations and the $5$ instances. We also use the following quantities to evaluate the performance of our algorithm:
\begin{itemize}
  \item feas$(x)$: For a given $x\in \R^n$, this corresponds to $\frac1m\#\{i:\; x\in D_i\}$.
  \item $\epsilon$-feas$(x)$: For a given $x\in \R^n$, this corresponds to $\frac1m\#\{i:\; \langle {\bm a}_i,x\rangle < b_i + \frac{10^{-5}m}4\}$ when each $D_i$ is a halfspace, and corresponds to
  $\frac1m\#\{i:\; \min\{\langle {\bm a}_i,x\rangle- b_i,\langle {\bm p}_i,x\rangle- q_i\} \!<\! \frac{10^{-5}m}4\}$ when each $D_i$ is a union of two halfspaces.
\end{itemize}
For each of the $5$ random instances, we take the maximum of feas$(x_0)$ and $\epsilon$-feas$(x_0)$ and the maximum of feas$(x_*)$ and $\epsilon$-feas$(x_*)$  over $5$ random initial points $x_0$, where $x_*$ is the approximate solution returned by our algorithm. We report in Table~\ref{t1} the average of these quantities over the $5$ random instances under the columns feas$_0$, $\epsilon$-feas$_0$, feas$_*$ and $\epsilon$-feas$_*$. One can see that our approach is able to identify a reasonably large (approximately) feasible subsystem with respect to $C$ (i.e., $\epsilon$-feas$_*\gtrsim\bar p$) in a reasonable period of time, even for large-scale problems. Moreover, we always obtain a larger feasible subsystem compared with that identified by the random initial points. Finally, we also observe that our algorithm is faster when each $D_i$ is convex.

\begin{table}[h]
\caption{Computational results for randomly generated MFS$_C$ problems}\label{t1}
\begin{center}
\begin{tabular}{|c|cc||cc|cc|cc|}\hline
  & $m$  & $\bar p$  & iter  & CPU  & feas$_0$   &$\epsilon$-feas$_0$ & feas$_*$   &$\epsilon$-feas$_*$ \\\hline
  \multirow{6}{*}{convex $D_i$}   &      & 0.5 &        8156 &   223.7 & 0.34 & 0.34 & 0.58 & 0.61   \\
                                  & 3000 & 0.6 &        2585 &    73.8 & 0.38 & 0.38 & 0.50 & 0.59   \\
                                  &      & 0.7 &         214 &     4.8 & 0.41 & 0.41 & 0.58 & 0.70   \\\cline{2-9}
                                  &      & 0.5 &        9638 &   807.1 & 0.37 & 0.37 & 0.56 & 0.59   \\
                                  & 5000 & 0.6 &        4107 &   386.7 & 0.39 & 0.39 & 0.51 & 0.59   \\
                                  &      & 0.7 &         246 &    15.3 & 0.42 & 0.42 & 0.57 & 0.70   \\\hline
  \multirow{6}{*}{Nonconvex $D_i$}&      & 0.5 &        9867 &   504.2 & 0.47 & 0.47 & 0.85 & 0.89 \\
                                  & 3000 & 0.6 &        9469 &   482.1 & 0.49 & 0.49 & 0.84 & 0.87 \\
                                  &      & 0.7 &       10372 &   541.0 & 0.52 & 0.52 & 0.84 & 0.87 \\\cline{2-9}
                                  &      & 0.5 &       14897 &  2268.3 & 0.51 & 0.51 & 0.85 & 0.89 \\
                                  & 5000 & 0.6 &       13637 &  2077.6 & 0.53 & 0.53 & 0.84 & 0.88 \\
                                  &      & 0.7 &       13164 &  2039.3 & 0.56 & 0.56 & 0.83 & 0.86 \\\hline
 \end{tabular}
\end{center}
\end{table}

\end{document}